\documentclass[11pt,reqno]{amsart}
\usepackage{a4wide}
\usepackage[numbers,sort&compress]{natbib}

\usepackage{color}
\usepackage[colorlinks, linkcolor=blue, citecolor=blue]{hyperref}

\allowdisplaybreaks[4]
\numberwithin{equation}{section}
\usepackage{mathrsfs}
\usepackage{amsfonts}
\usepackage{amsmath}
\usepackage{stmaryrd}
\usepackage{amssymb}
\usepackage{amsthm}
\usepackage{mathrsfs}
\usepackage{url}
\usepackage{amsfonts}
\usepackage{amscd}
\usepackage{indentfirst}
\usepackage{enumerate}
\usepackage{amsmath,amsfonts,amssymb,amsthm}
\usepackage{amsmath,amssymb,amsthm,amscd}
\usepackage{graphicx,mathrsfs}
\usepackage{appendix}
\usepackage[numbers,sort&compress]{natbib}
\usepackage{color}
\usepackage{verbatim}
 \usepackage{titlesec}

 \titleformat{\section}
   {\centering\Large\bfseries} 
  {\thesection}{1em}{}

\newcommand{\R}{\mathbb{R}}

\renewcommand{\theequation}{\arabic{section}.\arabic{equation}}
\renewcommand\thesection{\arabic{section}}
\makeatletter
\renewcommand\subsection{\@startsection{subsection}{2}{\z@}%
  {-3.25ex\@plus -1ex \@minus -.2ex}%
  {1.5ex \@plus .2ex}%
  {\normalfont\large\bfseries\itshape}}
\makeatother

\newtheorem{Thm}{Theorem}[section]
\newtheorem{Lem}[Thm]{Lemma}

\newtheorem{Prop}[Thm]{Proposition}

\newtheorem{Rem}[Thm]{Remark}

\begin{document}
\title[Concentrating solutions of nonlinear Schrödinger
systems with mixed interactions]
{Concentrating solutions of nonlinear Schrödinger
systems with mixed interactions}

\author{Qing Guo}
\address[Qing Guo]{College of Science, Minzu University of China, Beijing 100081, China}
\email{guoqing0117@163.com}

\author{Angela Pistoia}
\address[Angela Pistoia]{Dipartimento SBAI, Sapienza Universit\`a di Roma,
via Antonio Scarpa 16, 00161 Roma, Italy. }
\email{angela.pistoia@uniroma1.it}

  \author{Shixin Wen}
  \address[Shixin Wen]{School of Mathematics and Statistics, Central China Normal University, Wuhan 430079, China} \email{sxwen@mails.ccnu.edu.cn}

\thanks{A.Pistoia has been   partially supported by
the MUR-PRIN-20227HX33Z ``Pattern formation in nonlinear phenomena'' and
the INdAM-GNAMPA group. Q. Guo has been supported by the National Natural Science Foundation of China (Grant No. 12271539), and this research was also supported by the China Scholarship Council (CSC) during the visits of the first and third authors to Sapienza University of Rome.}

 \date{\today}

 \begin{abstract}
In this paper we study the existence of solutions to  nonlinear Schrödinger systems with  mixed couplings of
attractive and repulsive forces, which arise from the models
in Bose-Einstein condensates and nonlinear optics. In particular, we build solutions whose
first component has one bump  and the other components have several peaks
forming a regular polygon around  the single bump of the first component.
 \end{abstract}

\maketitle

\section{Introduction}
The  multispecies
Bose–Einstein condensates can be  described by the Gross-Pitaevskii system,
namely a system of  $q$  coupled  nonlinear Schr\"{o}dinger equations
$$
 i \hslash \partial_{t}\psi_{i}=-\frac{\hslash^2}{2m_{i}}\Delta \psi_{i}+ {V_i}(x)\psi_{i}-\mu_{i}|\psi_{i}|^{2}\psi_{i}-\sum\limits_{j=1\atop j\not=i}^q\beta_{ij}|\psi_{j}|^{2}\psi_{i} \ \hbox{in}\ \mathbb R^N,\ i=1,\dots,q
$$
where $t>0$ and the space dimension $N=1,2,3$.
Here,  $\hslash$ is the Planck constant and
$\psi_i=\psi_i(x,t)\in\mathbb C$ is the $i-$th component wave function,  with mass $m_i$ and external potential $V_i$.  $\mu_i$ is the {\em intraspecies} scattering length
and $\beta_{ij}=\beta_{ji}$ is the {\em interspecies} scattering length  for $i\not=j.$ The  interspecies interaction force is
 attractive if $\beta_{ij}>0$ while it is repulsive if $\beta_{ij}<0.$
 For more details on this model we refer to the book by Pitaevskii and Stringari \cite{pi-st}. \\
 Looking for standing waves $\psi_{i}(x,t)=
 e^{i\lambda_{i}t/\hslash}u_i(x)$
  of frequencies $\lambda_{i}$, we have that the real functions $u_i$ solve the nonlinear Schrödinger system
 \begin{equation}\label{ss}
 -\frac{\hslash^2}{2m_{i}}\Delta u_i+(\lambda_i+V_i(x))u_i=\mu_{i}u^{3}_i+\sum\limits_{j=1\atop j\not=i}^q\beta_{ij}u_j^{2}u_{i} \ \hbox{in}\ \mathbb R^N,\ i=1,\dots,q.
\end{equation}
 Systems like  \eqref{ss} in the presence of attractive
  self-interaction, i.e. $\mu_{i}>0$,  have been the subject of extensive
 mathematical studies in recent years. A huge number of results has been obtained in the
  the autonomous case, i.e. the potentials $V_{i}$ are positive constants. We refer the readers to the recent paper \cite{li-wei-wu} were an exhaustive list of references is given. On the other hand a few results are avalaible for the non-autonomous case.  In  a   singularly perturbed regime, i.e. $\hslash$ is very small with respect to the other parameters,   the existence of solutions to \eqref{ss} which concentrate as $\hslash\to0$  is quite expected.
 When  system \eqref{ss} has  only  two components, the concentration phenomena has been studied by
Lin and Wei \cite{lin-wei},  Montefusco,   Pellacci and Squassina \cite{MPS}, Pomponio \cite{pom},
Ikoma and Tanaka \cite{iko-tan} and
Byeon \cite{byeon}. The known results can be summerised in the existence of solutions whose both components concentrate at critical points (possibly the same) of the potentials or of a suitable linear combination of them in both attractive or repulsive regime.\\

Recently, Pellacci, Pistoia, Vaira and Verzini in \cite{ppvv} studied system \eqref{ss} with two equations  in a {\em partial} singularly perturbed regime. More precisely, they rewrite system \eqref{ss} as
\begin{equation}\label{p2}
\left\{\begin{aligned}
&-\Delta u+V(x)u=\mu_{1}u^{3}+\beta uv^{2}\ \hbox{in}\ \R^{N},
\\
&-\epsilon^{2}\Delta v+W(x)v=\mu_{2}v^{3}+\beta vu^{2}\ \hbox{in}\  \R^{N},
\end{aligned}\right.
\end{equation}
and  find   a solution $(u,v)$  in a repulsive regime, i.e. $\beta<0$, whose first component  looks like
a  solution to the single equation
$-\Delta Y+V(x)Y=\mu_{1}Y^{3}$ in $ \mathbb R^N$
 and whose second component concentrates at two   opposite points which collapse to the origin  as $\epsilon$ goes to 0.
Roughly speaking, the first component of \eqref{p2} does not concentrate and plays the  role of  an additional potential  in the singularly perturbed second equation,
so that the concentration behaviour of the second component   is ruled by the {\em shadow}  potential $\omega(x):=W(x)-\beta Y^{2}(x)$.
\\

In this paper we address a couple of natural questions.
\begin{itemize}
\item[(Q1)]{ \em Can we find a solution  whose second component concentrates at more than two points collapsing to the origin  as $\epsilon$ goes to 0?}
\item[(Q2)] {\em Can we extend this result  to a system with more than two equations?}
\end{itemize}
In the following, we will give a positive answer to both questions.\\

Let us introduce the necessary tools to state our main results.

Assume:
\begin{enumerate}
\item $V\in C^0\left(\mathbb{R}^N\right)$ is a radially symmetric function with
$ \inf _{\mathbb{R}^N} V(y)>0.
$ Moreover,  for every $f \in L^q\left(\mathbb{R}^N\right), 2 \leq q<+\infty$, there exists a unique solution $u \in W^{2, q}\left(\mathbb{R}^N\right)$ of the equation
$$
-\Delta u+V(y) u=f\ \hbox{in}\ \mathbb R^N,
$$
such that$$
\|u\|_{W^{2, q}\left(\mathbb{R}^N\right)} \le C\|f\|_{L^q\left(\mathbb{R}^N\right)},
$$
where $C$ only depends on $q$ and $N.$
\item There exists a unique positive radial solution $Y \in C^3\left(\mathbb{R}^N\right) \cap H^2\left(\mathbb{R}^N\right)$  to
$$
-\Delta Y+V(y) Y=\mu_1 Y^3\ \hbox{in}\ \mathbb R^N,
$$
which is non-degenerate in the space of even functions
$$
\begin{aligned}
& H_e^1\left(\mathbb{R}^N\right) :=\left\{u \in H^1\left(\mathbb{R}^N\right): u\left(y_1, \ldots, y_i, \ldots y_N\right)=u\left(y_1, \ldots,-y_i, \ldots, y_N\right),\,i=1, \ldots, N\right\}.
\end{aligned}
$$

\item  $W\in C^3\left(\mathbb{R}^N\right)$ is a radially symmetric function with {$ \inf _{\mathbb{R}^N} W(y)>0$}.  For sake of simplicity, we also assume that $  W,D^\alpha W \in L^\infty\left(\mathbb{R}^N\right),$  $\alpha=1,2,3.$

\end{enumerate}

Throughout the paper,  we will   assume $\beta<0$. Moreover,  we set
\begin{equation}
    \label{omega}\omega(y):=W(y)-\beta Y^{2}(y)\quad\hbox{and}\quad\omega_0:=\omega(0).\end{equation}
Since
 $\omega_0 >0, $ we define
 $U_{\mu_2}\in H^1\left(\mathbb{R}^N\right) $ as the radial solution of
\begin{equation}\label{U-1}
-\Delta U_{\mu_2}+\omega_0U_{\mu_2}=\mu_2 U_{\mu_2}^3,\quad  U_{\mu_2}>0,\quad \text { in $\mathbb{R}^N$.}
\end{equation}
It is useful to remind that it satisfies
$$   \lim_{|y|\rightarrow+\infty} U_{\mu_{2}}(y)\,e^{\sqrt{\omega_0}|y|}\,|y|^{\frac{N-1}{2}}=C_0>0,\quad  \lim_{|y|\rightarrow+\infty}\frac{U_{\mu_{2}}^{\prime}(y)}{U_{\mu_{2}}(y)}=-1.
$$

We introduce the rotation in the $(x_1,x_2)-$plane:
\begin{equation}\label{roto}\mathcal{R}_{\theta}:=\begin{pmatrix}
        \cos\theta&\sin\theta&0\\
        -\sin\theta&\cos\theta&0\\
        0&0&I_{(N-2)\times (N-2)}
    \end{pmatrix}
\end{equation}
and   the Banach space
$$H_{s}^{2}\left(\mathbb{R}^N\right):= \left\{u \in H ^2\left(\mathbb{R}^N\right)\Big|~{u} \text { is even in each variable},\  u\left(y\right)=
u\left(\mathcal{R}_{\frac{2\pi}{k}}    y\right)\right\}.
$$
First we   consider the system  \eqref{p2} with only   two components.\\
Our result reads as follows.

\begin{Thm}\label{mt1}
Assume $ \Delta\omega(0)<0$.
 For any even integer $k$ there exists $\epsilon_k>0$ such that for any $\epsilon\in(0,\epsilon_k)$, \eqref{p2} has a solution $(u_\epsilon,v_\epsilon)\in H_{s}^{2}\left(\mathbb{R}^N\right)\times H_{s}^{2}\left(\mathbb{R}^N\right)$ as
$$   u_\epsilon(y)= Y(y)+\Theta(y)\ \hbox{and}\  v_\epsilon(y)=  \sum\limits_{j=1}^k U_{\mu_{2}}\left(y-\frac{P_{j}}{\epsilon}\right)+\Xi(y)
$$
where
\begin{equation*}
P_{j}=\rho_{\epsilon}\left(\cos\frac{2(j-1)\pi}{k},~~\sin\frac{2(j-1)\pi}{k},~~0\right) \ \hbox{with}  \lim_{\epsilon\rightarrow0^{+}}\frac{\rho_{\epsilon}}{\epsilon|\ln\epsilon|}=\frac{1}{\sqrt{\omega_0} \sin{\frac{\pi}{k}}}
\end{equation*}
and  $\Theta,\Xi$ are small functions in suitable norms.

\end{Thm}

Next we consider the system with $m+1$ components in a {\em symmetric} shape
\begin{equation}
    \label{pm}
 \left\{   \begin{aligned}
&-\Delta u+V(y)u=\mu_{1}u^{3}+\beta u\left(\sum\limits_{i=1}^{m}v_{i}^{2}\right),\quad \hbox{in}\ \mathbb R^{N},
\\
&
-\epsilon^{2}\Delta v_i+W(y)v_i=\mu_{2} v_{i}^{3}+\beta v_{i}u^{2}+\alpha\sum\limits_{l=1\atop l\neq i}^{m}v_{i}\,v_{l}^{2},\quad \hbox{in}\ \mathbb R^{N},\quad i=1,\cdots,m.
\end{aligned}\right.
\end{equation}
 Our second result reads as follows.
\begin{Thm}\label{mt}
Assume  $\alpha\cdot \Delta\omega(0)<0.$
 For any even integer $k$ there exists $\epsilon_k>0$ such that for any $\epsilon\in(0,\epsilon_k)$, \eqref{pm} has a solution $(u_\epsilon,v_{1,\epsilon},\cdots,v_{m,\epsilon})\in H^2_s(\mathbb R^N)\times H^2_s(\mathbb R^N)$ as
$$
u_\epsilon(y)= Y(y)+\Theta(y)$$
and
$$v_{1,\epsilon}(y)=  \sum\limits_{j=1}^k U_{\mu_{2}}\left(y-\frac{P_{j}}{\epsilon}\right)+\Xi(y),\  v_{i,\epsilon}(y)=v_{1,\epsilon} \left(\mathcal{R}_{\frac{2\pi(i-1)}{mk}} y\right),\   i=2,\cdots,m,
$$
where
\begin{equation*}
P_{j}=\rho_{\epsilon}\left(\cos\frac{2(j-1)\pi}{k},~~\sin\frac{2(j-1)\pi}{k},~~0\right) \ \hbox{with}  \lim_{\epsilon\rightarrow0^{+}}\frac{\rho_{\epsilon}}{\epsilon|\ln\epsilon|}=\frac{1}{2\sqrt{\omega_0} \sin{\frac{\pi}{mk}}}
\end{equation*}
and  $\Theta,\Xi$ are small functions in suitable norms.

\end{Thm}

\begin{Rem}\rm
Our results state the existence of a solution whose  components $v_1,\dots,v_m$ have $k$ peaks
forming a regular $k-$ polygon  around the origin, being a critical point of the {\em shadow}  potential $\omega$ introduced in \eqref{omega}.
We observe that while the interaction between the first component $u$ and all the other components $v_i$ is  repulsive, i.e. $\beta<0$,
the interaction between the components $v_i$ and $v_j$ can be either repulsive, i.e. $\alpha<0$ or attractive, i.e. $\alpha>0$ depending
on the nature of the origin as a critical point of $\omega$.  In particular, if $0$ is a minimum point of $\omega$ then  $\Delta \omega(0)>0$ and  there exists a solution in a fully repulsive regime, i.e.  $\alpha<0$ and $\beta<0$. On the other hand, if $0$ is a maximum point of $\omega$ then $\Delta \omega(0)<0$  and there exists a solution in a mixed attractive  and repulsive regime, i.e. $\alpha>0$ and $\beta<0$. \\
We observe that  if $V$  is radially non-decresing near
the origin then  $Y$ has a  strict maximum at the origin itself and $\Delta Y(0)<0$ (e.g. \cite[Remark 1.2]{ppvv}). Therefore if $\beta<0$ and $|\beta|$ is large enough, the shadow potential $\omega$ satisfies $\Delta \omega(0)=\Delta W(0)-2\beta Y(0)\Delta Y(0)<0$.
\end{Rem}
 \begin{Rem}\rm 
When the potentials $V=\lambda_1$ and $W=\lambda_2$  are constants 
the shadow potential $\omega(y)=\lambda_2-\beta Y^2(y)$  has only one maximum point at the origin which satisfies
$\Delta \omega(0)=-2\beta Y(0)\Delta Y(0)<0.$  
In the case of only two components we find a solution which resembles the one found
 by Lin and Wei in     \cite[Theorem 1.2]{lin-wei-fisica}  for a similar system in a small repulsive regime.
 On the other hand, in the case of more components  Theorem \ref{mt} is new and   suggests that    the result  in  \cite{lin-wei-fisica}
could be extended   to   a  system   having {\em mixed interactions}.
 \end{Rem}

\begin{Rem}\rm
Inspired by our results, one could study the partial singularly perturbed system with $m=r+s$ components
\begin{equation}\label{conj-s}
\left\{
\begin{aligned}
&-\Delta u_i+V_i(x)u_i=\mu_i u_i^3+\sum\limits_{j=1\atop j\not=i}^r a_{ij}u_iu_j^2+\sum\limits_{\ell=1}^s \beta_{i\ell}u_iv_\ell^2\ \hbox{in}\ \mathbb R^N,\ i=1,\dots,r,\\
&-\epsilon^2\Delta v_\ell+W_\ell(x)v_\ell=\nu_\ell v_\ell^3+\sum\limits_{\kappa=1\atop \kappa\not=\ell}^s b_{\ell \kappa}v_\ell v_\kappa^2+\sum\limits_{i=1}^r \beta_{ \ell i} v_\ell u_i^2\ \hbox{in}\ \mathbb R^N,\ \ell=1,\dots,s.\\
\end{aligned}\right.
\end{equation}
By mimicking the results described above, it would be interesting to explore the possibility to build   a solution to \eqref{conj-s} whose first $r-$components $(u_1,\dots,u_r)$ look like the solution of the system (if it exists)
$$-\Delta u_i+V_i(x)u_i=\mu_i u_i^3+\sum\limits_{j=1\atop j\not=i}^r a_{ij}u_iu_j^2\quad \hbox{in}\ \mathbb R^N,\ i=1,\dots,r,$$
while the second $s-$components $(v_1,\dots,v_s)$ concentrate around one or more points as $\epsilon\to0.$ 
In particular, the first $r-$components of \eqref{conj-s} do not concentrate and should play the  role of  an additional potential  in the singularly perturbed system with the remaining $s-$components. It would be challenging to detect the shadow potential which rules
the concentration behaviour of the $s-$components.
\end{Rem}

\begin{Rem}\rm The proof of our results relies on a classical Lyapunov-Schmidt procedure and follows the strategy introduced and developed in \cite{ppvv,pv}.
That is why we will omit many details of the proof and we will only show what can not be immediately deduced by \cite{ppvv,pv}.\end{Rem}

{\em Notation.}\, In what follows we agree that notation $f=\mathcal O(g)$ or $f\lesssim g$ stands for
$|f|\le C|g|$ for some
$ C >0$   uniformly with respect all the variables involved,
 unless it is not specified.

\section{Proof of Theorem \ref{mt1} and Theorem \ref{mt}}

\subsection{Preliminaries}
We will find a solution $(u,v_1,\cdots,v_m)$ to $\eqref{pm}$  such that (see \eqref{roto})
$$v_{i}(y)=v \left(\widehat{\mathcal{R}_{i}}\,y\right),\ \widehat{\mathcal{R}_{i}}:=\mathcal{R}_{\frac{2\pi(i-1)}{mk}} \ \hbox{for any}\ i=1,\cdots,m,
$$
where $(u,v)\in H^2_s(\mathbb R^N)\times H^2_s(\mathbb R^N)$ solves the non-local problem
\begin{equation}
    \label{pmr}
 \left\{   \begin{aligned}
&-\Delta u+V(y)u=\mu_{1}u^{3}+\beta u \sum\limits_{i=1}^{m}v^2 \left(\widehat{\mathcal{R}_{i}}\,y\right) \ \hbox{in}\ \mathbb R^{N},
\\
&
-\epsilon^{2}\Delta v +W(y) v=\mu_{2} v^{3}+\beta vu^{2}+\alpha v \sum\limits_{i=2}^{m}v^2 \left(\widehat{\mathcal{R}_{i}}\,y\right) \ \ \hbox{in}\ \mathbb R^{N}.
\end{aligned}\right.
\end{equation}

 We consider the scaled version of problem  \eqref{pmr}
\begin{equation}\label{SP}
 \left\{   \begin{aligned}
&-\Delta u+V(y)u=\mu_{1}u^{3}+\beta u \sum\limits_{i=1}^{m}\left(v\circ  \widehat{\mathcal{R}_{i}}\right)^2\left(\frac{y}{\epsilon}\right) \ \hbox{in}\ \mathbb R^{N},
\\
&
-\Delta v +W(\epsilon y) v=\mu_{2} v^{3}+\beta vu^{2}(\epsilon y)+\alpha v \sum\limits_{i=2}^{m} \left(v\circ  \widehat{\mathcal{R}_{i}}\right)^2\left(\frac{y}{\epsilon}\right) \ \ \hbox{in}\ \mathbb R^{N}.
\end{aligned}\right.
\end{equation}
We point out that the system \eqref{p2} with only two components is a particular case of the system \eqref{pm} taking $\alpha=0.$ That is why from now on we only consider the system \eqref{pm} and we agree that the case $m=1$ corresponds to the choice $\alpha=0.$\\
 We introduce the Banach spaces
$$H_{V }^2\left(\mathbb{R}^N\right):=\left\{u\in H^2_s\left(\mathbb{R}^N\right)\ :\ u\left(y\right)=u\left(\widehat{\mathcal{R}_{l}}    y\right),\,l=2,\cdots,m, \ \displaystyle{\int_{\R^N}} V(y)\,u^{2}dy<+\infty\right\},$$
and
$$H_{W_{\epsilon}}^2\left(\mathbb{R}^N\right):=\left\{v\in H^2_s\left(\mathbb{R}^N\right):\displaystyle{\int_{\R^N}} W(\epsilon y)\,v^{2}dy<+\infty\right\},$$ equipped with the norms
$$\|u\|_{V}:=\left(\displaystyle{\int_{\R^N}}\sum_{|\alpha|=2}\left|D^\alpha u\right|^2+|\nabla u|^{2}+V(y)\,u^{2}dy\right)^{\frac{1}{2}},$$
and
$$\|v\|_{W_{\epsilon}}:=\left(\displaystyle{\int_{\R^N}}\sum_{|\alpha|=2}\left|D^\alpha v\right|^2+|\nabla v|^{2}+W(\epsilon y)\,v^{2}dy\right)^{\frac{1}{2}}.$$
We set $X:=H_{V}^{2}\left(\R^N\right)\times H^{2}_{W_{\epsilon}}\left(\R^N\right)$ equipped with the norm $\|(u,v)\|=\|u\|_{V}+\|v\|_{W_{\epsilon}}.$\\
We agree that the case  $m=1$  corresponds to $\alpha=0.$ We also point out that when $m=1$, the condition $u\left(y\right)=u\left(\widehat{\mathcal{R}_{l}}    y\right)$, for $l=2,\cdots,m$, is overloaded as soon as $u\in H^2_s\left(\mathbb{R}^N\right)$.\\

We look for  a solution $(u,v)\in X$ to \eqref{SP} as
\begin{equation}\label{solu}
   u(y)= \underbrace{Y(y)+\beta \Phi_{\epsilon}(y)}_{:=\Upsilon_\epsilon(y)}+\phi(y),\quad v(y)=\underbrace{U_{\rho,{\epsilon}}(y)+\beta \Psi_{\epsilon}(y)}_{:=\Theta_\epsilon(y)}+\psi(y),
\end{equation}
where   $k\ge 2$ is an even integer,
$$U_{\rho,{\epsilon}}(y):= \sum\limits_{j=1}^k U_{\mu_{2}, P_{\epsilon,j}}(y),\quad  U_{\mu_{2}, P_{\epsilon,j}}(y):=U_{\mu_{2}}\left(y-\frac{P_{j}}{\epsilon}\right)
   $$
with the peaks
\begin{equation}\label{pj}  P_{j}=\rho \left(\cos\frac{2(j-1)\pi}{k},~~\sin\frac{2(j-1)\pi}{k},~~0\right)\ \hbox{and}\ \rho=d\epsilon|\ln\epsilon|\ \hbox{for some}\ d>0.\end{equation}
The functions $\Phi_{\epsilon}$ and $\Psi_{\epsilon}$ are suitable correction terms whose existence and properties
are given  in Lemma \ref{Phi}. Moreover, the remainder term $\left(\phi,\psi\right)$ belongs to the space
\begin{equation*}
   \left\{(\phi,\psi)\in X\,\Big|\,\displaystyle{\int_{\R^N}}\psi(y)\,Z_\epsilon(y)\,dy=0\right\},
    \end{equation*}
where
\begin{equation}\label{z-1}
    \begin{split}
        Z_\epsilon(y)&:=\epsilon\sum\limits_{j=1}^k \frac{\partial}{\partial \rho}U_{\mu_{2},P_{\epsilon,j}}\left(y\right)\\&=-\sum\limits_{j=1}^k \left({\cos\frac{2(j-1)\pi}{k}\frac{\partial}{\partial y_1}U_{\mu_{2},P_{\epsilon,j}}\left(y\right)+\sin\frac{2(j-1)\pi}{k}\frac{\partial}{\partial y_2}U_{\mu_{2},P_{\epsilon,j}}\left(y\right)}\right)\\
        &:=\sum\limits_{j=1}^kZ_{\epsilon,j}(y)
    \end{split}
\end{equation}
is a solution of the linear equation
\begin{equation}\label{z-ep}
    -\Delta Z_\epsilon
    +\omega_0\, Z_\epsilon=3\mu_2 \left(\sum\limits_{j=1}^k U_{\mu_{2},P_{\epsilon,j}}^{2}(y)\,Z_{\epsilon,j}(y)\right)\ \hbox{in}\ \mathbb R^N.
\end{equation}
We are going to use the classical Lyapunov-Schmidt procedure: we shall adjust the unknown parameter $d=d(\epsilon)$ in \eqref{pj} to get a genuine solution to \eqref{SP}. As it is usual this choice will be made in the last step of the procedure and will be deduced by the balance  between the interactions between the bubbles and the potential (whose order is  always $\epsilon\rho$) and the interactions between the bubbles themselves which strongly depend on the minimum distance between all the peaks. In  the case of only two components,
the shortest distance among the peaks   is
\begin{equation}\label{pk}
\min\limits_{j=2,\cdots,k}\frac{\left|P_1-P_j\right|}{\epsilon}=\frac{\left|P_1-P_2\right|}{\epsilon}=\frac{2\rho}{\epsilon}\,{\sin{\frac{\pi}{k}}},
\end{equation}
while in the case of more components the minimal distance among all the peaks 
\begin{equation}\label{pij}
P_{ji}:=\widehat{\mathcal{R}_{i}}^{-1}    P_{j}=\widehat{\mathcal{R}_{-i}}    P_{j},\  j=1,\cdots,k,\ i=2,\cdots,m,\end{equation}
is
\begin{equation*} \min\limits_{j=1,\cdots,k;\,i=2,\cdots,m}\frac{\left|P_j-P_{ji}\right|}{\epsilon}=\frac{\left|P_1-P_{12}\right|}{\epsilon}=\frac{2\rho}{\epsilon}{\sin{\frac{\pi}{mk}}}.
\end{equation*}
We would like to point out that if $i=1$, $P_{j1}=P_{j}$ for $j=1,\cdots,k$.
More precisely, in Lemma \ref{important}   if  $\alpha=0$ the balance reads as
\begin{equation}\label{bala1}\epsilon\rho\sim  e^{^{-2\sqrt{\omega_0}\,\frac{\rho}{\epsilon}\sin{\frac{\pi}{ k}}}}\left(\frac{\rho}{\epsilon}\right)^{-\frac{N-1}{2}}\end{equation}
while   if $\alpha\not =0$ it reads as
\begin{equation}\label{bala2}
\begin{aligned}
&\epsilon\rho\sim  e^{^{-4\sqrt{\omega_0}\,\frac{\rho}{\epsilon}\sin{\frac{\pi}{mk}}}}\left(\frac{\rho}{\epsilon}\right)^{-\frac{1}{2}}&\ \hbox{if}\ N=2,\\
&\epsilon\rho\sim e^{^{-4\sqrt{\omega_0}\,\frac{\rho}{\epsilon}\sin{\frac{\pi}{mk}}}}\left(\frac{\rho}{\epsilon}\right)^{-{2}}\ln{\left(\frac{\rho}{\epsilon}\right)}
& \ \hbox{if}\  N=3.
\end{aligned}\end{equation}

From now on, we agree that  \eqref{bala1} and \eqref{bala2} hold true when $\alpha=0$ and $\alpha\not=0,$ respectively.

\subsection{The correction of the ansatz}
We define the correction terms $\Phi_\epsilon$ and $\Psi_\epsilon$ in \eqref{solu}  following the ideas in  \cite[Section 2]{ppvv}.
 \begin{Lem}\label{Phi}${}$\\
  \begin{enumerate}[(i)]
  \item   There exists a unique function
$\Phi_\epsilon\in H_{V}^{2}(\R^N)$ verifying
\begin{equation}\label{phi-ep}
    -\Delta \Phi_\epsilon+\left(V(y)-3\mu_{1}Y^{2}(y)\right)\Phi_\epsilon={Y(y)\left(\sum\limits_{i=1}^m U^{2}_{\rho,\epsilon}\left(\widehat{\mathcal{R}_{i}}   \frac{y}{\epsilon}\right)\right)}.
\end{equation}
such that
\begin{equation}
    \label{es-1}
\left\|\Phi_\epsilon\right\|_{C^{0,\frac{1}{2}}(\R^N)}\lesssim\left\|\Phi_\epsilon\right\|_{V}\lesssim\epsilon^{\frac{N}{2}}
\end{equation}
and for any $q>N$,
\begin{equation}
    \label{es-3}
    \left\|\Phi_\epsilon\right\|_{C^{1,1-\frac{N}{q}}(\R^N)}\lesssim\left\|\Phi_\epsilon\right\|_{W^{2,q}(\R^N)}\lesssim\epsilon^{\frac{N}{q}}.
\end{equation}
 \item There exists a unique function
$\Psi_\epsilon \in H^{2}_{W_\epsilon}(\R^N)$ verifying
\begin{equation}\label{psi-ep}
    -\Delta\Psi_\epsilon+\left(\omega_0-3\mu_{2}\sum\limits_{j=1}^k U^{2}_{\mu_{2},P_{\epsilon,j}}\left(y\right)\right)\Psi_\epsilon=2\beta \Phi_{\epsilon}(0)Y(0){\left(\sum\limits_{j=1}^k U_{\mu_{2},P_{\epsilon,j}}\left(y\right)\right)}.
\end{equation}
such that
\begin{equation}
    \label{es-4}
\left\|\Psi_\epsilon\right\|_{C^{0,\frac{1}{2}}(\R^N)}\lesssim\left\|\Psi_\epsilon\right\|_{H^{2}(\R^N)}\lesssim\epsilon^{\frac{N}{2}}.
\end{equation}
Moreover, there exist $\gamma\in (0,\sqrt{\omega_0}) $ and $R_0>0$ such that
\begin{equation}
    \label{es-5}
\left|\Psi_\epsilon(y)\right|\lesssim\epsilon^{\frac{N}{2}}\left(\sum\limits_{j=1}^k e^{-\gamma\left|y-\frac{P_{j}}{\epsilon}\right|}\right),\quad\text{for any $|y|\geq R_0$.}
\end{equation}
\end{enumerate}
\end{Lem}
\begin{proof}
 $(i)$ and $(ii)$ can be proved as in \cite[Lemma 2.1]{ppvv} and \cite[Lemma 2.3]{ppvv}, respectively. The symmetry properties of $\Phi_\epsilon$ and $\Psi_\epsilon$ can be deduced by the symmetries of
  the R.H.S. in \eqref{phi-ep} and  the R.H.S. in \eqref{psi-ep}. Indeed, the function
 $$f(y):=Y(y)\left(\sum\limits_{i=1}^m U^{2}_{\rho,\epsilon}\left(\widehat{\mathcal{R}_{i}}   \frac{y}{\epsilon}\right)\right)=Y(y)\left(\sum\limits_{i=1}^m \left(\sum\limits_{j=1}^{k} U_{\mu_{2}}\left(\frac{y-P_{ji}}{\epsilon}\right)\right)^{2}\,\right)$$
with $P_{ji}$ defined in \eqref{pij},
  is   even and satisfies
$$f(y)=f\left(\mathcal{R}_{\frac{2\pi
}{k}}    y\right)\quad\hbox{and}\quad  f(y)=f\left(\widehat{\mathcal{R}_{l}}    y\right),\,l=2,\cdots,m,$$
while the function
 $$g(y):=2\beta \Phi_{\epsilon}(0)Y(0){\left(\sum\limits_{j=1}^k U_{\mu_{2},P_{\epsilon,j}}\left(y\right)\right)}$$
  is   even and satisfies
$$g(y)=g\left(\mathcal{R}_{\frac{2\pi
}{k}}    y\right).$$
\end{proof}
\begin{Rem}\label{after}
Since $\nabla \Phi_\epsilon(0)=0$, from \eqref{es-3}, we deduce that for any $q>N$,
\begin{equation}\label{regular}
    \left|\Phi_\epsilon(y)-\Phi_\epsilon(0)\right|\lesssim\left\|\Phi_\epsilon\right\|_{C^{1,1-\frac{N}{q}}(\R^N)}|y|^{2-\frac{N}{q}}\lesssim\epsilon^{\frac{N}{q}}|y|^{2-\frac{N}{q}}.
\end{equation}

\end{Rem}

\subsection{The reduction process}

Let us introduce  the one dimensional space
$$
 {K}:= \left\{\left(0, cZ_{\varepsilon}\right)\ :\ c\in\mathbb R\right\} \subset L^2(\R^N)\times L^2(\R^N),$$
where $Z_{\varepsilon}(y)$ is defined as \eqref{z-1}, its orthogonal
$$ {K}^{\perp}  :=\left\{(f, g) \in L^2\left(\mathbb{R}^N\right) \times L^2\left(\mathbb{R}^N\right): \int_{\mathbb{R}^N} g(y) Z_{\varepsilon}(y) dy=0\right\},$$
and the projections
$$
 {\Pi}: L^2\left(\mathbb{R}^N\right) \mapsto  {K}, \quad \text { and } \quad {\Pi}^{\perp}: L^2\left(\mathbb{R}^N\right) \mapsto  {K}^{\perp}.
$$
Then taking account the ansatz in \eqref{solu}, we can rewrite \eqref{SP} as the equivalent system
\begin{equation}\label{equiv-sys1}
  {\Pi}^\perp\left\{\mathcal{L}(\phi, \psi)-\mathcal{E}-\mathcal{N}(\phi, \psi)\right\}=0,
 \end{equation}
 \begin{equation}\label{equiv-sys2} {\Pi} \left\{\mathcal{L}(\phi, \psi)-\mathcal{E}-\mathcal{N}(\phi, \psi)\right\}=0.
  \end{equation}
Here the linear operator $\mathcal{L}=\left(\mathcal{L}_1, \mathcal{L}_2\right)$ is defined by
\begin{equation}\label{linear-op}
    \begin{aligned}
& \mathcal{L}_1(\phi, \psi):=-\Delta \phi+V(y) \phi-\left(3 \mu_1 \Upsilon_{\epsilon}^2(y)+\beta \sum\limits_{i=1}^{m}\Theta_{\epsilon}^2\left(\widehat{\mathcal{R}}_{i}   \frac{y}{\epsilon}\right)\right) \phi(y)\\
&\quad\,\quad\,\quad\,\quad\,\quad-2 \beta \Upsilon_{\epsilon}(y)\left(\sum_{i=1}^{m}\Theta_{\epsilon}\left(\widehat{\mathcal{R}}_{i}   \frac{y}{\epsilon}\right)  \psi\left(\widehat{\mathcal{R}}_{i}   \frac{y}{\epsilon}\right)\right), \\
& \mathcal{L}_2(\phi, \psi):=-\Delta \psi+W(\epsilon y) \psi-\Big(3 \mu_2 \Theta_{\epsilon}^2(y)+\beta \Upsilon_{\epsilon}^2(\epsilon y)\Big) \psi(y)-2 \beta \Theta_{\epsilon}(y) \Upsilon_{\epsilon}(\epsilon y) \phi(\epsilon y)\\
&\quad\,\quad\,\quad\,\quad\,\quad-2\alpha\Theta_{\epsilon}(y)\left(\sum_{i=2}^{m}\Theta_{\epsilon}\left(\widehat{\mathcal{R}}_{i}   {y}\right)  \psi\left(\widehat{\mathcal{R}}_{i}   {y}\right)\right)-\alpha\psi(y)\left(\sum_{i=2}^{m}\Theta^{2}_{\epsilon}\left(\widehat{\mathcal{R}}_{i}   {y}\right) \right).
\end{aligned}
\end{equation}
The error term $\mathcal{E}=\left(\mathcal{E}_1, \mathcal{E}_2\right)$ is defined by

\begin{equation}\label{error}
    \begin{aligned}
\mathcal{E}_1:= & 3 \mu_1 \beta^2 \Phi_{\epsilon}^2(y)Y(y) +\mu_1 \beta^3 \Phi_{\epsilon}^3(y)+\beta^3 Y(y) \left(\sum\limits_{i=1}^{m}\Psi^{2}_{\epsilon}\left(\widehat{\mathcal{R}_{i}}   \frac{y}{\epsilon}\right)\right)\\
&+2 \beta^2 Y(y)\left(\sum\limits_{i=1}^{m} U_{\rho,{\epsilon}}\left(\widehat{\mathcal{R}_{i}}   \frac{y}{\epsilon}\right) \Psi_{\epsilon}\left(\widehat{\mathcal{R}_{i}}   \frac{y}{\epsilon}\right)\right)  \\
&+\beta^2 \Phi_{\epsilon}(y) \left(\sum\limits_{i=1}^{m}\left(U_{\rho,{\epsilon}}^2\left(\widehat{\mathcal{R}_{i}}   \frac{y}{\epsilon}\right)+2 \beta U_{\rho,{\epsilon}}\left(\widehat{\mathcal{R}_{i}}   \frac{y}{\epsilon}\right) \Psi_{\epsilon}\left(\widehat{\mathcal{R}_{i}}   \frac{y}{\epsilon}\right)+\beta^2\Psi_{\epsilon}^{2}\left(\widehat{\mathcal{R}_{i}}   \frac{y}{\epsilon}\right)\right)\right), \\
\mathcal{E}_2:= & \left(\omega_0-\omega(\epsilon y)\right) \Theta_{\epsilon}(y)+2 \beta^2 U_{
\rho,{\epsilon}}(y)\Big(Y(\epsilon y) \Phi_{\epsilon}(\epsilon y)-Y(0)\Phi_{\epsilon}(0) \Big) \\
& +\beta^3 \Phi_{\epsilon}^2(\epsilon y) \Theta_{\epsilon}(y)+2 \beta^3 \Psi_{\epsilon}(y)Y(\epsilon y) \Phi_{\epsilon}(\epsilon y)+3 \mu_2 \beta^2 U_{\rho,{\epsilon}}(y) \Psi_{\epsilon}^2(y)+\mu_2 \beta^3 \Psi_{\epsilon}^3(y) \\
&{ +\mu_2\left(U_{\rho,{\epsilon}}^3(y)-\sum_{j=1}^{k}U_{\mu_{2},P_{\epsilon,j}}^3(y)\right)}+3\mu_2 \beta \Psi_{\epsilon}(y)\left(U_{\rho,{\epsilon}}^2(y)-\sum_{j=1}^{k}U_{\mu_{2},P_{\epsilon,j}}^2(y)\right)\\
&{+\alpha\Theta_{\epsilon}(y)\left(\sum\limits_{i=2}^{m}\Theta^{2}_{\epsilon}\left(\widehat{\mathcal{R}_{i}}    y)\right)\right)}.
\end{aligned}
\end{equation}
Finally, the nonlinear term $\mathcal{N}=\left(\mathcal{N}_1, \mathcal{N}_2\right)$ is defined by
\begin{equation}
\label{perturbed}
\begin{aligned}
 \mathcal{N}_1(\phi, \psi):=&\mu_1 \phi^2(y)\Big(3 \Upsilon_{\epsilon}(y)+\phi(y)\Big)+\beta \Upsilon_{\epsilon}(y) \left(\sum\limits_{i=1}^{m}\psi^2\left(\widehat{\mathcal{R}_{i}}   \frac{y}{\epsilon}\right)\right)\\
&+\beta \phi(y) \left(\sum\limits_{i=1}^{m}\left(2 \Theta_{\epsilon}\left(\widehat{\mathcal{R}_{i}}   \frac{y}{\epsilon}\right)\psi\left(\widehat{\mathcal{R}_{i}}   \frac{y}{\epsilon}\right)+\psi^{2}\left(\widehat{\mathcal{R}_{i}}   \frac{y}{\epsilon}\right)\right )\right),\\
 \mathcal{N}_2(\phi, \psi):=&\mu_2 \psi^2(y)\Big(3 \Theta_{\epsilon}(y)+\psi(y)\Big)+\beta \psi(y) \phi(\epsilon y)\Big(2 \Upsilon_{\epsilon}(\epsilon y)+\phi(\epsilon y)\Big)+\beta \Theta_{\epsilon}(y) \phi^2(\epsilon y)\\
 &+\alpha\Theta_{\epsilon}(y)\left(\sum\limits_{i=2}^{m}\psi^{2}\left(\widehat{\mathcal{R}_{i}}    y\right)\right)+\alpha\psi(y)\sum\limits_{i=2}^{m}\left(2\Theta_{\epsilon}\left(\widehat{\mathcal{R}_{i}}    y\right)\psi\left(\widehat{\mathcal{R}_{i}}    y\right)+\psi^{2}\left(\widehat{\mathcal{R}_{i}}    y\right)\right).
\end{aligned}
\end{equation}

\subsection{Solving  equation (\ref{equiv-sys1})}
The proof is standard and relies on a classical contraction mapping argument as in \cite[Proposition 3.4]{ppvv}.

\begin{Prop}\label{contraction} There exists $\epsilon_{0}>0$ such that for any $\epsilon\in(0,\epsilon_0)$, there exists { a unique solution     $(\phi_\epsilon,\psi_\epsilon)\in X$ of \ref{equiv-sys1} such that
    $\displaystyle{\int_{\R^N}}\psi_\epsilon(y)\,Z_\epsilon(y)\,dy=0.$}
Moreover, if $\alpha=0$ $(m=1)$
\begin{equation}
    \label{al-zero}
    \|(\phi_\epsilon, \psi_\epsilon)\|\lesssim \epsilon^{2}|\ln{\epsilon}|^{2},
\end{equation}
and  if $\alpha\neq 0$ $(m\geq 2)$,
\begin{equation}
    \label{alnot-zero}\|(\phi_\epsilon, \psi_\epsilon)\|\lesssim e^{-2\sqrt{\omega_{0}}\frac{\rho}{\epsilon}\sin{\frac{\pi}{mk}}}\left(\frac{\rho}{\epsilon}\right)^{-\frac{N-1}{2}}\lesssim\left\{\begin{aligned}&\epsilon|\ln\epsilon|^\frac14\,\quad\hbox{if}\ N=2,\\
&\frac{\epsilon|\ln\epsilon|^\frac12}{\left(\ln|\ln\epsilon|\right)^{\frac12}}\,\quad\hbox{if}\ N=3.\\
\end{aligned}\right.
\end{equation}
\end{Prop}
Two main ingredients in the proof of Proposition \ref{contraction} are the invertibility of the linear operator $\mathcal{L}$ and the  size of the error $\mathcal{E}$.\\
 The first ingredient concerns the linear theory   and can be proved as in  \cite[Lemma 3.1]{ppvv}
\begin{Lem}\label{coer}
   There exists a constant $c_0>0$ and $\epsilon_{0}>0$ such that for any $\epsilon\in(0,\epsilon_0)$, we have
    $$\left\| {\Pi}^{\perp}\mathcal{L}(\phi, \psi)\right\|_{L^2(\R^N)\times L^2(\R^N)}\geq c_0\left\|(\phi, \psi)\right\|,$$ for any
    $(\phi,\psi)\in X$ such that
    $\displaystyle{\int_{\R^3}}\psi(y)\,Z_\epsilon(y)\,dy=0.$
    \end{Lem}

The size of the error is different from the one found in \cite[Proposition 3.3]{ppvv}.  It is clear that it is due to the shortest distance of the  points in the configuration, {see \eqref{pk} if $\alpha=0$ $(m=1)$ and $\eqref{pij}$ if $\alpha\neq 0$ $(m\geq 2)$, respectively}.
This fact affects the rate of the error as it is stated in the following result.

\begin{Lem}
    \label{error size}There exists $\epsilon_{0}>0$ such that for any $\epsilon\in(0,\epsilon_0)$
$$\|\mathcal{E}_{1}\|_{L^2(\R^N)} \lesssim \epsilon^{{N}},
$$
and
$$\|\mathcal{E}_{2}\|_{L^2(\R^N)} \lesssim
   \epsilon^{2}|\ln{\epsilon}|^{2}\ \hbox{if $\alpha=0$,}\quad  \|\mathcal{E}_{2}\|_{L^2(\R^N)} \lesssim  e^{-2\sqrt{\omega_{0}}\frac{\rho}{\epsilon}\sin{\frac{\pi}{mk}}}\left(\frac{\rho}{\epsilon}\right)^{-\frac{N-1}{2}}\ \hbox{if $\alpha\not=0$.}
$$
\\
Then  if $\alpha =0$ $(m=1)$
 \begin{equation}
      \label{he-1}
\|\mathcal{E}\|_{L^2(\R^N)\times L^2(\R^N)}\lesssim \epsilon^{2}|\ln\epsilon|^2,\end{equation}
and if $\alpha\not=0$ $(m\geq 2)$,
  \begin{equation}
      \label{he-2}
\|\mathcal{E}\|_{L^2(\R^N)\times L^2(\R^N)}\lesssim\left\{\begin{aligned}& \epsilon|\ln\epsilon|^\frac14 ,\quad\hbox{if}\ N=2,\\
& \frac{\epsilon|\ln\epsilon|^\frac12}{(\ln|\ln\epsilon|)^\frac12}, \quad \hbox{if}\ N=3.\\
\end{aligned}\right.\end{equation}
\end{Lem}

\begin{proof}
The proof follows as in the proof of  \cite[Proposition 3.3]{ppvv}. However, for sake of completeness, we repeat the arguments here.

   By \eqref{error} and   Lemma \ref{Phi},
\begin{align*}
\|\mathcal{E}_{1}\|_{L^2(\R^N)}\lesssim&\left\|\Phi_{\epsilon}^2(y)Y(y)\right\|_{L^2(\R^N)} +\left\|\Phi_{\epsilon}^3(y)\right\|_{L^2(\R^N)} +\left\|Y(y) \left(\sum\limits_{i=1}^{m}\Psi^{2}_{\epsilon}\left(\widehat{\mathcal{R}_{i}}   \frac{y}{\epsilon}\right)\right)\right\|_{L^2(\R^N)}\\
&+\left\|Y(y)\left(\sum\limits_{i=1}^{m} U_{\rho,{\epsilon}}\left(\widehat{\mathcal{R}_{i}}   \frac{y}{\epsilon}\right) \Psi_{\epsilon}\left(\widehat{\mathcal{R}_{i}}   \frac{y}{\epsilon}\right)\right) \right\|_{L^2(\R^N)}+\left\|\Phi_{\epsilon}\sum\limits_{i=1}^{m}\Psi_{\epsilon}^{2}\left(\widehat{\mathcal{R}_{i}}   \frac{y}{\epsilon}\right)\right\|_{L^2(\R^N)}\\
&+\left\|\Phi_{\epsilon} \left(\sum\limits_{i=1}^{m}U_{\rho,{\epsilon}}^2\left(\widehat{\mathcal{R}_{i}}   \frac{y}{\epsilon}\right)\right)\right\|_{L^2(\R^N)} +\left\|\Phi_{\epsilon} \sum\limits_{i=1}^{m}U_{\rho,{\epsilon}}\left(\widehat{\mathcal{R}_{i}}   \frac{y}{\epsilon}\right) \Psi_{\epsilon}\left(\widehat{\mathcal{R}_{i}}   \frac{y}{\epsilon}\right)\right\|_{L^2(\R^N)}\\
=&O\left(\epsilon^{{N}}\right).
\end{align*}
Moreover, 
\begin{align}
&\|\mathcal{E}_{2}\|_{L^2(\R^N)}\nonumber\\\lesssim& \underbrace{\left\|\left(\omega_0-\omega(\epsilon y)\right) U_{
\rho,{\epsilon}}(y)\right\|_{L^2(\R^N)}}_{:=M_1}+\underbrace{\left\|\left(\omega_0-\omega(\epsilon y)\right) \Psi_{{\epsilon}}(y)\right\|_{L^2(\R^N)}}_{:=M_2}\nonumber\\
&+\underbrace{\left\|U_{
\rho,{\epsilon}}(y)\Big(Y(\epsilon y) \Phi_{\epsilon}(\epsilon y)-Y(0)\Phi_{\epsilon}(0) \Big)\right\|_{L^2(\R^N)}}_{:=M_3} +\underbrace{\left\|\left(U^3_{\rho,{\epsilon}}(y)-\sum_{j=1}^{k}U_{\mu_{2},P_{\epsilon,j}}^3(y)\right)\right\|_{L^2(\R^N)}}_{:=M_4}\nonumber\\& +\underbrace{\left\|\Psi_{\epsilon}(y)\left(U_{\rho,{\epsilon}}^2(y)-\sum_{j=1}^{k}U_{\mu_{2},P_{\epsilon,j}}^2(y)\right)\right\|_{L^2(\R^N)}}_{:=M_5}+\underbrace{\left\|\Theta_{\epsilon}(y)\left(\sum\limits_{i=2}^{m}\Theta^{2}_{\epsilon}\left(\widehat{\mathcal{R}_{i}}    y\right)\right)\right\|_{L^2(\R^N)}}_{:=M_6}\nonumber\\
& +\underbrace{\left\|\Phi_{\epsilon}^2(\epsilon y) \Theta_{\epsilon}(y)\right\|_{L^2(\R^N)}}_{:=M_7}+\underbrace{\left\|\Psi_{\epsilon}(y)Y(\epsilon y) \Phi_{\epsilon}(\epsilon y)\right\|_{L^2(\R^N)}}_{:=M_8} \nonumber\\&+\underbrace{\left\| U_{\rho,{\epsilon}}(y) \Psi_{\epsilon}^2(y)\right\|_{L^2(\R^N)}}_{:=M_9}+\underbrace{\left\|\Psi_{\epsilon}^3(y)\right\|_{L^2(\R^N)}}_{:=M_{10}}.\label{e2}
\end{align}
Let us estimate each term in the right side of \eqref{e2}.
Note that $y=0$ is the critical point of the function $\omega$ given in \eqref{omega}, then 
$$\left|\omega_0-\omega(\epsilon y)\right|\lesssim \left|\epsilon y\right|^{2},$$
and
\begin{align*}
    M_{1}&\lesssim \left(\displaystyle{\int_{\R^{N}}}\epsilon^{4}|y|^{4}\left(\sum_{i=1}^{k} U^{2}_{\mu_{2}}\left(y-\frac{P_{{i}}}{\epsilon}\right)\right)dy\right)^{\frac{1}{2}}\nonumber\\
    &\lesssim~\epsilon^{2}\left(\displaystyle{\int_{\R^{N}}}|y|^{4}\,U^{2}_{\mu_{2}}\left(y-\frac{P_{1}}{\epsilon}\right)dy\right)^{\frac{1}{2}}\lesssim\epsilon^{2}\left(\frac{\rho}{\epsilon}\right)^{2}=\rho^{2}.
\end{align*}
By \eqref{es-1} in Lemma \ref{Phi},
\begin{align*}
    M_2&\lesssim \left(\displaystyle{\int_{\R^{N}}}\epsilon^{4}|y|^{4}\left(\Phi_{\epsilon}^{2}(0)\sum_{i=1}^{k} \Psi^{2}\left(y-\frac{P_{{i}}}{\epsilon}\right)\right)dy\right)^{\frac{1}{2}}\nonumber\\
    &\lesssim \epsilon^{2}\left\|\Phi_{\epsilon}\right\|_{L^{\infty}(\R^N)}\left(\displaystyle{\int_{\R^{N}}}|y|^{4}\Psi^{2}\left(y-\frac{P_{{1}}}{\epsilon}\right)dy\right)^{\frac{1}{2}}\nonumber\\
    &\lesssim \epsilon^{2+\frac{N}{2}}\left(\frac{\rho}{\epsilon}\right)^{2}=\epsilon^{\frac{N}{2}}\rho^{2}=o\left(\rho^{2}\right).
\end{align*}
Now,
\begin{align*}
  M_{3}\leq &\underbrace{\left\|U_{
\rho,{\epsilon}}(y)Y(\epsilon y)\Big( \Phi_{\epsilon}(\epsilon y)-\Phi_{\epsilon}(0) \Big)\right\|_{L^2(\R^N)}}_{:=M_{31}}
+\underbrace{\left\|U_{
\rho,{\epsilon}}(y)\Big(Y(\epsilon y) -Y(0)\Big)\Phi_{\epsilon}(0) \right\|_{L^2(\R^N)}}_{:=M_{32}}=o(\rho^2),
\end{align*}
since by \eqref{regular}
\begin{align*}
    M_{31}&\lesssim\left(\displaystyle{\int_{\R^{N}}}\epsilon^{4}|y|^{2\left(2-\frac{N}{q}\right)}Y^{2}(\epsilon y)\left(\sum_{i=1}^{k} U_{\mu_{2}}^{2}\left(y-\frac{P_{{i}}}{\epsilon}\right)\right)dy\right)^{\frac{1}{2}}\\
    &\lesssim\epsilon^{2}\left\|Y\right\|_{L^{\infty}(\R^N)}\left(\displaystyle{\int_{\R^{N}}}|y|^{2\left(2-\frac{N}{q}\right)}U_{\mu_2}^{2}\left(y-\frac{P_{{1}}}{\epsilon}\right)dy\right)^{\frac{1}{2}}\\
    &\lesssim\epsilon^{2}\left(\frac{\rho}{\epsilon}\right)^{2-\frac{N}{q}}=\epsilon^{\frac{N}{q}}\rho^{2-\frac{N}{q}}=o\left(\rho^{2}\right).
\end{align*}
and similarly, by
$$\left|Y(\epsilon y) -Y(0)\right|\lesssim \epsilon^{2}|y|^{2},$$
\begin{align*}
    M_{32}&\lesssim\epsilon^{2}\left\|\Phi_\epsilon\right\|_{L^{\infty}(\R^N)}\left(\displaystyle{\int_{\R^{N}}}|y|^{4}U^{2}\left(y-\frac{P_{{1}}}{\epsilon}\right)dy\right)^{\frac{1}{2}}\\
    &\lesssim\epsilon^{2+\frac{N}{2}}\left(\frac{\rho}{\epsilon}\right)^{2}=o\left(\rho^{2}\right).
\end{align*}

{Next}, by \eqref{U-1} and \eqref{bala1}, we obtain that if $\alpha=0$,
\begin{align*}
M_4&\lesssim\left(\displaystyle{\int_{\R^N}} \left|\sum\limits_{i,j=1}^k\sum_{i\neq j}U^{2}_{\mu_2,P_{\epsilon,j}}(y)U_{\mu_2,P_{\epsilon,i}}(y)\right|^{2}dy\right)^{\frac{1}{2}}
\nonumber\\&\lesssim \left(\sum\limits_{i=2}^k\displaystyle{\int_{\R^N}} U_{\mu_{2}}^{4}\left(x\right)U_{\mu_{2}}^{2}\left(x-\frac{P_{i}-P_{1}}{\epsilon}\right)dx\right)^{\frac{1}{2}}\nonumber\\
&\lesssim e^{-\sqrt{\omega_{0}}\frac{|P_{2}-P_{1}|}{\epsilon}}\left(\frac{|P_{2}-P_{1}|}{\epsilon}\right)^{-\frac{N-1}{2}}\\
&\lesssim e^{-2\sqrt{\omega_{0}}\frac{\rho}{\epsilon}\sin{\frac{\pi}{k}}}\left(\frac{\rho}{\epsilon}\right)^{-\frac{N-1}{2}}=o\left(\rho^{2}\right).
\end{align*}
On the other hand, {by \eqref{bala2}, we deduce that if $\alpha\neq0$ and $\sin{\frac{\pi}{k}}>2\sin{\frac{\pi}{mk}}$},
\begin{align*}
  M_4& \lesssim e^{-2\sqrt{\omega_{0}}\frac{\rho}{\epsilon}\sin{\frac{\pi}{k}}}\left(\frac{\rho}{\epsilon}\right)^{-\frac{N-1}{2}}=o\left(\rho^{2}\right).
\end{align*}

Similarly, combining \eqref{bala1} and \eqref{bala2} with Lemma \ref{lma2.2} respectively, we derive that {both in the cases of $\alpha= 0$ and $\alpha\neq 0$},
\begin{align*}
    M_5&\lesssim\left\|\Psi_\epsilon
    \right\|_{L^{\infty}(\R^N)}\left(\sum\limits_{i=2}^k\displaystyle{\int_{\R^N}} U_{\mu_{2}}^{2}\left(x\right)U_{\mu_{2}}^{2}\left(x-\frac{P_{i}-P_{1}}{\epsilon}\right)dx\right)^{\frac{1}{2}}\nonumber\\
    &\lesssim
    \begin{cases}
        \epsilon\,e^{-2\sqrt{\omega_{0}}\frac{\rho}{\epsilon}\sin{\frac{\pi}{k}}}\left(\frac{\rho}{\epsilon}\right)^{-\frac{1}{4}},\quad&\text{if $N=2$,}\nonumber\\\epsilon^{\frac{3}{2}}\,e^{-2\sqrt{\omega_{0}}\frac{\rho}{\epsilon}\sin{\frac{\pi}{k}}}\left(\frac{\rho}{\epsilon}\right)^{-1}\left(\ln{\left|{\frac{\rho}{\epsilon}}\right|}\right)^{\frac{1}{2}},\quad&\text{if $N=3$,}
    \end{cases}\nonumber\\
    &=o\left(\rho^{2}\right).
    \end{align*}

Finally, {if $\alpha\neq0$, we have the extra term}
\begin{align*}
    \|M_{6}\|_{L^2(\R^N)}\lesssim&\left(\displaystyle\int_{\R^N}\left(U_{\rho,\epsilon}(y)+\beta\Psi_\epsilon(y)\right)^{2}\,\left(U_{\rho,\epsilon}\left(\widehat{\mathcal{R}_{2}}    y\right)+\beta\Psi_\epsilon\left(\widehat{\mathcal{R}_{2}}    y\right)\right)^{4}dy\right)^{\frac{1}{2}}\nonumber\\
    \lesssim&\underbrace{\left(\displaystyle\int_{\R^N}U_{\rho,\epsilon}^{2}(y)\,U_{\rho,\epsilon}^{4}\left(\widehat{\mathcal{R}_{2}}    y\right)dy\right)^{\frac{1}{2}}}_{:=M_{61}}+\underbrace{\left(\displaystyle\int_{\R^N}U_{\rho,\epsilon}^{2}(y)\,U_{\rho,\epsilon}^{3}\left(\widehat{\mathcal{R}_{2}}    y\right)\,\Psi_{\epsilon}\left(\widehat{\mathcal{R}_{2}}    y\right)dy\right)^{\frac{1}{2}}}_{:=M_{62}}\nonumber\\
&+\underbrace{\left(\displaystyle\int_{\R^N}U_{\rho,\epsilon}(y)\,\Psi_{\epsilon}\left(y\right)\,U_{\rho,\epsilon}^{4}\left(\widehat{\mathcal{R}_{2}}    y\right)dy\right)^{\frac{1}{2}}}_{:=M_{63}}.
\end{align*}
Since
\begin{align*}
    \|M_{61}\|_{L^2(\R^N)}\lesssim&\left(\displaystyle\int_{\R^N}U_{\mu_{2},P_{\epsilon,1}}^{2}\left( y\right)\,U_{\mu_{2},P_{\epsilon,12}}^{4}\left( y\right)dy\right)^{\frac{1}{2}}\\
    \lesssim&e^{-\sqrt{\omega_0}\,\frac{|P_{12}-P_1|}{\epsilon}}\left(\frac{|P_{12}-P_1|}{\epsilon}\right)^{-\frac{N-1}{2}}\lesssim e^{-2\sqrt{\omega_{0}}\frac{\rho}{\epsilon}\sin{\frac{\pi}{mk}}}\left(\frac{\rho}{\epsilon}\right)^{-\frac{N-1}{2}},
\end{align*}
\begin{align*}
\|M_{62}\|_{L^2(\R^N)}\lesssim&    \|\Psi_{\epsilon}\|^{\frac{1}{2}}_{L^{\infty}(\R^N)}\left(\displaystyle\int_{\R^N}U_{\mu_{2},P_{\epsilon,1}}^{2}\left( y\right)\,U_{\mu_{2},P_{\epsilon,12}}^{3}\left( y\right)dy\right)^{\frac{1}{2}}\\
\lesssim& \epsilon^{\frac{N}{4}}\,e^{-\sqrt{\omega_0}\,\frac{|P_{12}-P_1|}{\epsilon}}\left(\frac{|P_{12}-P_1|}{\epsilon}\right)^{-\frac{N-1}{2}}=o\left(e^{-2\sqrt{\omega_{0}}\frac{\rho}{\epsilon}\sin{\frac{\pi}{mk}}}\left(\frac{\rho}{\epsilon}\right)^{-\frac{N-1}{2}}\right).
\end{align*}
and
\begin{align*}
    \|M_{63}\|_{L^2(\R^N)}=& \left(\displaystyle\int_{\R^N}U_{\rho,\epsilon}(y)\,U_{\rho,\epsilon}^{2}\left(\widehat{\mathcal{R}_{2}}    y\right)\,\Psi_{\epsilon}\left(y\right)\,U_{\rho,\epsilon}^{2}\left(\widehat{\mathcal{R}_{2}}    y\right)dy\right)^{\frac{1}{2}}\\
    \lesssim&\underbrace{\left(\displaystyle\int_{\R^N}U^{2}_{\rho,\epsilon}(y)\,U_{\rho,\epsilon}^{4}\left(\widehat{\mathcal{R}_{2}}    y\right)dy\right)^{\frac{1}{4}}}_{:=G_1}\underbrace{\left(\displaystyle\int_{\R^N}\Psi^{2}_{\epsilon}\left(y\right)\,U_{\rho,\epsilon}^{4}\left(\widehat{\mathcal{R}_{2}}    y\right)dy\right)^{\frac{1}{4}}}_{:=G_2},
\end{align*}
with
\begin{align*}
   G_1\lesssim&\left(\displaystyle\int_{\R^N}U^{2}_{\mu_{2},P_{\epsilon,1}}(y)\,U_{\mu_{2},P_{\epsilon,12}}^{4}\left( y\right)dy\right)^{\frac{1}{4}}
    \lesssim e^{-\frac{\sqrt{\omega_0}}{2}\,\frac{|P_{12}-P_1|}{\epsilon}}\left(\frac{\rho}{\epsilon}\right)^{-\frac{N-1}{4}},
\end{align*}
and by \eqref{es-5}, since $0<\gamma<\sqrt{\omega_{0}}$, then
\begin{align*}
    &G_2\\\lesssim&\left(\displaystyle\int_{\R^N}4\beta^{2}\Phi_{\epsilon}^{2}(0)Y^{2}(0)\left(\sum_{j=1}^{k}\Psi^{2}\left(y-\frac{P_j}{\epsilon}\right)\right)\,\left(\sum_{q=1}^{k}U_{\mu_{2},P_{\epsilon,q2}}^{4}\left(y\right)\right)dy\right)^{\frac{1}{4}}\\
    \lesssim&\|\Phi_{\epsilon}\|^{\frac{1}{2}}_{L^{\infty}(\R^N)}\left(\displaystyle\int_{\R^N}\Psi^{2}\left(y-\frac{P_1}{\epsilon}\right)\,U_{\mu_{2}}^{4}\left(y-\frac{P_{12}}{\epsilon}\right)dy\right)^{\frac{1}{4}}\\
    \lesssim&\epsilon^{\frac{N}{4}}\left(\displaystyle\int_{\R^N}\Psi^{2}\left(x\right)\,U_{\mu_{2}}^{4}\left(x-\frac{P_{12}-P_1}{\epsilon}\right)dx\right)^{\frac{1}{4}}\\
    =&\epsilon^{\frac{N}{4}}\left(\displaystyle\int_{B_{\frac{|P_{12}-P_1|}{2\epsilon}}}\Psi^{2}\left(x\right)\,U_{\mu_{2}}^{4}\left(x-\frac{P_{12}-P_1}{\epsilon}\right)+\displaystyle\int_{B^{c}_{\frac{|P_{12}-P_1|}{2\epsilon}}}\Psi^{2}\left(x\right)\,U_{\mu_{2}}^{4}\left(x-\frac{P_{12}-P_1}{\epsilon}\right)\right)^{\frac{1}{4}}\\
    \lesssim&\epsilon^{\frac{N}{4}}\left(e^{-2\sqrt{\omega_0}\,\frac{|P_{12}-P_1|}{\epsilon}}\displaystyle\int_{\R^N}\Psi^{2}+e^{-\gamma\,\frac{|P_{12}-P_1|}{\epsilon}}\displaystyle\int_{\R^N}U_{\mu_{2}}^{4}\right)^{\frac{1}{4}}\\\lesssim&\epsilon^{\frac{N}{4}}e^{-\frac{\gamma}{4}\,\frac{|P_{12}-P_1|}{\epsilon}}.
\end{align*}
Therefore, combining the estimates of $G_1,\,G_2$ with \eqref{bala2}, we have
\begin{align*}
    { \|M_{63}\|_{L^2(\R^N)}}\lesssim&\epsilon^{\frac{N}{4}}e^{-\left(\frac{\sqrt{\omega_0}}{2}+\frac{\gamma}{4}\right)\frac{|P_{12}-P_1|}{\epsilon}}\left(\frac{\rho}{\epsilon}\right)^{-\frac{N-1}{4}}\\
=&o\left(e^{-2\sqrt{\omega_{0}}\frac{\rho}{\epsilon}\sin{\frac{\pi}{mk}}}\left(\frac{\rho}{\epsilon}\right)^{-\frac{N-1}{2}}\right).
\end{align*}
Then from the estimates of $M_{61}$-$M_{63}$, we conclude that
$$
    \|M_6\|_{L^2(\R^N)}\lesssim e^{-2\sqrt{\omega_{0}}\frac{\rho}{\epsilon}\sin{\frac{\pi}{mk}}}\left(\frac{\rho}{\epsilon}\right)^{-\frac{N-1}{2}}.
$$
To estimate the last four terms in \eqref{e2}, by  Lemma \ref{Phi}
\begin{align*}
    &\left\|M_7\right\|_{L^2(\R^N)}\lesssim\left\|\Phi_{\epsilon}\right\|^{2}_{L^\infty(\R^N)}\lesssim\epsilon^N=o\left(\rho^{2}\right),\quad\nonumber\\&\left\|M_8\right\|_{L^2(\R^N)}\lesssim\left\|Y\right\|_{L^\infty(\R^N)}\left\|\Phi_{\epsilon}\right\|_{L^\infty(\R^N)}\left\|\Psi_{\epsilon}\right\|_{H^2(\R^N)}\lesssim\epsilon^N=o\left(\rho^{2}\right),\nonumber\\
    &\left\| M_9\right\|_{L^2(\R^N)} \lesssim\left\|\Psi_{\epsilon}\right\|^{2}_{H^2(\R^N)}\lesssim\epsilon^N=o\left(\rho^{2}\right),\\
    &
\left\|M_{10}\right\|_{L^2(\R^N)}\lesssim\left\|\Psi_{\epsilon}\right\|^{3}_{H^2(\R^N)}\lesssim\epsilon^{\frac{3N}{2}}=o\left(\rho^{2}\right).
\end{align*}

Finally, we  combine all the estimates of $M_{1}$-$M_{10}$ and we get \eqref{he-1} and \eqref{he-2}  taking into account the choice in \eqref{bala1} and \eqref{bala2}, respectively.

\end{proof}

   \subsection{Solving equation (\ref{equiv-sys2})}
   Since equation \ref{equiv-sys1} is solved, we know that there exists $(\phi, \psi):=(\phi_\epsilon, \psi_\epsilon)$
   and $c_\epsilon\in\mathbb R$ such that
   \begin{equation}\label{first}
    \mathcal{L}(\phi, \psi)-\mathcal{E}-\mathcal{N}(\phi, \psi) =c_{\epsilon}\begin{pmatrix}   0 \\ Z_\epsilon   \\ \end{pmatrix} ,
\end{equation}

We have to find $d=d_\epsilon$ in \eqref{pj} such that $c_\epsilon=0$ in \eqref{first}. Since
  \begin{equation}\label{ceps}
 c_{\epsilon}={\displaystyle{\int_{\mathbb R^N} } \left(\mathcal{L}_2(\phi, \psi)-\mathcal{E}_2-\mathcal{N}_2(\phi, \psi)\right)Z_\epsilon\,dy \over
 \displaystyle{\int_{\mathbb R^N} }   Z_\epsilon^2\,dy},
\end{equation}
we are going to estimate the R.H.S. of \eqref{ceps}.

\begin{Lem}
    \label{important} If $\epsilon>0$ is small enough,
    \begin{itemize}
    \item if $\alpha=0$,
    \begin{align}\label{c-alpha}
        &\displaystyle{\int_{\mathbb{R}^N}}\left( \mathcal{L}_2(\phi, \psi)-\mathcal{E}_2-\mathcal{N}_2(\phi, \psi)\right)Z_\epsilon dy\nonumber\\=&-\left(\Delta\omega(0)A\epsilon\rho+B_{1}e^{-2\sqrt{\omega_{0}}\frac{\rho}{\epsilon}\sin{\frac{\pi}{k}}}\left(\frac{\rho}{\epsilon}\right)^{-\frac{N-1}{2}}\right)\left(1+o(1)\right),
    \end{align}
    for some positive constants $A$ and $B_1$;
    \item if $\alpha\neq0$,
    \begin{align}\label{conclusion}
        &\displaystyle{\int_{\mathbb{R}^N}}\left( \mathcal{L}_2(\phi, \psi)-\mathcal{E}_2-\mathcal{N}_2(\phi, \psi)\right)Z_\epsilon dy\nonumber\\=&\begin{cases}
            -\left(\Delta\omega(0)A\epsilon\rho+\alpha B_{2}e^{-4\sqrt{\omega_{0}}\sin{\frac{\pi}{mk}}\frac{\rho}{\epsilon}}\left(\frac{\rho}{\epsilon}\right)^{-\frac{1}{2}}\right)\left(1+o(1)\right),\ \hbox{if}\ N=2,\\
-\left(\Delta\omega(0)A\epsilon\rho+\alpha B_{2}e^{-4\sqrt{\omega_{0}}\sin{\frac{\pi}{mk}}\frac{\rho}{\epsilon}}\left(\frac{\rho}{\epsilon}\right)^{-{2}}\ln{\left(\frac{\rho}{\epsilon}\right)}\right)\left(1+o(1)\right),\ \hbox{if}\  N=3,
        \end{cases}
    \end{align}
      for some positive constants $A$ and $B_2$.
      \end{itemize}
\end{Lem}
 \begin{proof}
 First, let us estimate the leading term
\begin{align*}
&\displaystyle{\int_{\mathbb{R}^N}} \mathcal{E}_2 Z_{\epsilon} dy\\= & \underbrace{\displaystyle{\int_{\mathbb{R}^N}}\left(\omega_0-\omega(\epsilon y)\right) \Theta_{\epsilon} Z_{\epsilon} dy}_{:=F_1}+\underbrace{2\beta^2 \displaystyle{\int_{\mathbb{R}^N}}U_{\rho,\epsilon}\Big(Y(\epsilon y) \Phi_{\epsilon}(\epsilon y)-Y(0)\Phi_{\epsilon}(0)\Big) Z_{\epsilon}dy}_{:=F_2} \\
&+\underbrace{\mu_2 \displaystyle{\int_{\mathbb{R}^N}}\left(U_{\rho,\epsilon}^3-\sum\limits_{j=1}^kU^{{3}}_{\mu_{2},P_{\epsilon,j}}\right) Z_{\epsilon} dy}_{:=F_3} \\
&+\underbrace{3 \mu_2 \beta \displaystyle{\int_{\mathbb{R}^N}}\left(U_{\rho,\epsilon}^2-\sum\limits_{j=1}^kU^{{2}}_{\mu_{2},P_{\epsilon,j}}\right) \Psi_{\epsilon} Z_{\epsilon} dy}_{:=F_4} \\
&+\underbrace{\beta^3\displaystyle{\int_{\mathbb{R}^N}}\Phi_{\epsilon}^2(\epsilon y) \Theta_{\epsilon} Z_{\epsilon} dy}_{:=F_5}+\underbrace{2\beta^3 \displaystyle{\int_{\mathbb{R}^N}}\Psi_{\epsilon} Y(\epsilon y) \Phi_{\epsilon}(\epsilon y) Z_{\epsilon} dy}_{:=F_6} \\
&+\underbrace{3\mu_2 \beta^2 \displaystyle{\int_{\mathbb{R}^N}} U_{\rho,\epsilon} \Psi_{\epsilon}^2 Z_{\epsilon} dy}_{:=F_7}+\underbrace{\mu_2 \beta^3 \displaystyle{\int_{\mathbb{R}^N}}\Psi_{\epsilon}^3 Z_{\epsilon} dy}_{:=F_8}\\
&{+\underbrace{\alpha\displaystyle{\int_{\mathbb{R}^N}}\Theta_{\epsilon}(y)\left(\sum\limits_{i=2}^{m}\Theta^{2}_{\epsilon}\left(\widehat{\mathcal{R}_{i}}    y)\right)\right)Z_{\epsilon} dy}_{:=F_{9}}}.
\end{align*}
In the following, we will estimate each term in the right side of the above equality. Firstly, using similar arguments as the estimates in \cite[Lemma 4.2]{ppvv}, we have
\begin{align*}
    F_1=&\underbrace{\displaystyle{\int_{\R^N}}(\omega(0)-\omega(\epsilon y)) \left(\sum\limits_{j=1}^{k}U_{\mu_{2},P_{\epsilon,j}}(y)Z_{{\epsilon,j}}(y)\right)  dy}_{:=F_{11}}\\&+\underbrace{\displaystyle{\int_{\R^N}}(\omega(0)-\omega(\epsilon y)) \left(\sum\limits_{i,j=1}^k\sum_{i\neq j} U_{\mu_{2},P_{\epsilon,j}}(y)Z_{{\epsilon,i}}(y)\right)  dy}_{:=F_{12}}\\
    &+\underbrace{\beta \displaystyle{\int_{\R^N}}(\omega(0)-\omega(\epsilon y)) \Psi_{\epsilon}(y)\, Z_{{\epsilon}}(y)\,dy}_{:={F_{13}}}.
\end{align*}
Since $y=0$ is the critical point of $\omega$, then from \eqref{z-1}, we get that
\begin{align*}
    F_{11}=&-\frac{1}{2} \displaystyle{\int_{\R^N}}\left\langle D^2 \omega(0) \epsilon y, \epsilon y\right\rangle \left(\sum\limits_{j=1}^{k}U_{\mu_{2},P_{\epsilon,j}}(y)Z_{{\epsilon,j}}(y)\right)  dy \\
& +{O}\left(\epsilon^3\displaystyle{\int_{\R^N}}\left|y\right|^3 \left(\sum\limits_{j=1}^{k}U^{2}_{\mu_{2},P_{\epsilon,j}}(y)\right) dy\right) \\
= & \frac{k}{2}\displaystyle{\int_{\R^N}}\left\langle D^2 \omega(0)\left(\epsilon x+P_{1}\right), \left(\epsilon x+P_{1}\right)\right\rangle U_{\mu_{2}}(x)\partial_1 U_{\mu_{2}}(x) dx+o\left(\epsilon \rho\right) \\
=&\frac{k}{2}\left(\displaystyle{\int_{\R^N}}\left\langle D^2 \omega(0) \epsilon x, P_{1}\right\rangle U_{\mu_{2}} \partial_1 U_{\mu_{2}} dx+\displaystyle{\int_{\R^N}}\left\langle D^2 \omega(0) P_{1}, \epsilon x\right\rangle U_{\mu_{2}} \partial_1 U_{\mu_{2}}dx\right)+o\left(\epsilon \rho\right)\\=&k\,\epsilon\rho\displaystyle{\int_{\R^N}}\left\langle D^2 \omega(0)P_{0}, x\right\rangle U_{\mu_{2}}(x) \partial_1 U_{\mu_{2}}(x) dx+o\left(\epsilon \rho\right)\\
=&k\,\epsilon\rho\left(-\partial_{11}\omega(0)\right)\left(-\displaystyle{\int_{\R^N}}\,\frac{x_{1}^{2}}{|x|}\,U_{\mu_{2}}(x) \,U^{\prime}_{\mu_{2}}(x) dx\right)+o\left(\epsilon \rho\right)\\
=& -\partial_{11}\omega(0)\,\tilde{b}\,\epsilon\rho+o\left(\epsilon \rho\right)\\
=&-\Delta\omega(0)A\epsilon\rho+o\left(\epsilon \rho\right),
\end{align*}
here $P_0:=(1,0,0)$ and  $A:=c_n\tilde{b}>0$, where the positive constant $c_n$ verifies $\partial_{11}\omega(0)=c_n\Delta\omega(0)$ (since $\omega$ is a radial symmetric function), and $$\tilde{b}:=k\left(-\displaystyle{\int_{\R^N}}\,\frac{x_{1}^{2}}{|x|}\,U_{\mu_{2}}(x) \,U^{\prime}_{\mu_{2}}(x) dx\right)>0,\quad\text{since $U_{\mu_{2}}$ is radially decreasing. }$$
Next, we claim that the terms $F_{12}$ and $F_{13}$ are of higher order with respect to $\epsilon \rho$. Indeed, a direct computation yields that
\begin{align*}
    F_{12}&=k\displaystyle{\int_{\R^N}}(\omega(0)-\omega(\epsilon y)) U_{\mu_{2},P_{\epsilon,1}}(y)\left(\sum\limits_{i=2}^kZ_{{\epsilon,i}}(y)\right)  dy\\&\lesssim\left(\displaystyle{\int_{\R^N}}\left|\omega(0)-\omega(\epsilon y)\right|^{2}U_{\mu_{2},P_{\epsilon,1}}(y)\right)^{\frac{1}{2}}\left(\displaystyle{\int_{\R^N}}U_{\mu_{2},P_{\epsilon,1}}(y)\left(\sum\limits_{i=2}^kU^{2}_{\mu_{2},P_{\epsilon,i}}(y)\right) \right)^{\frac{1}{2}}\\
    &\lesssim\left(\displaystyle{\int_{\R^N}}\left|\epsilon y\right|^{4}U_{\mu_{2},P_{\epsilon,1}}(y)\right)^{\frac{1}{2}}\left(\displaystyle{\int_{\R^N}}U_{\mu_{2}}(x)\left(\sum\limits_{i=2}^kU^{2}_{\mu_{2}}\left(x-\frac{P_{i}-P_{1}}{\epsilon}\right)\right) dx\right)^{\frac{1}{2}}\\
&\lesssim\epsilon^{2}\left(\frac{\rho}{\epsilon}\right)^{2}e^{-\sqrt{\omega_{0}}\frac{\rho}{\epsilon}\sin{\frac{\pi}{k}}}\left(\frac{\rho}{\epsilon}\right)^{-\frac{N-1}{4}}=o\left(\epsilon \rho\right),
\end{align*}
and
\begin{align*}
    |F_{13}|&\lesssim\left(\displaystyle{\int_{\R^N}}\left|\omega(0)-\omega(\epsilon y)\right|^{2}Z^{2}_{\epsilon}(y)\right)^{\frac{1}{2}}\|\Psi_\epsilon\|_{L^{2}(\R^N)}\\
    &\lesssim\left(\sum\limits_{j=1}^k\displaystyle{\int_{\R^N}}\left|\epsilon y\right|^{4}U^{2}_{\mu_{2}}\left(y-\frac{P_{j}}{\epsilon}\right) dy\right)^{\frac{1}{2}}\|\Psi_\epsilon\|_{H^{1}(\R^N)}\\&\lesssim\epsilon^{2}\left(\frac{\rho}{\epsilon}\right)^{2}\,\epsilon^{\frac{N}{2}}=o\left(\epsilon \rho\right).
\end{align*}

Therefore, we conclude that
$$F_{1}=-\Delta\omega(0)A\epsilon\rho+o\left(\epsilon \rho\right),$$
for some positive constant $A$.

In addition, we show that the terms $F_{2}$ and $F_{5}$-$F_{8}$ are of higher order with respect to $\epsilon \rho$. In fact,
\begin{align*}
{F_2}=&\underbrace{\displaystyle{\int_{\mathbb{R}^N}}\left(\sum_{j=1}^{k}U_{\mu_{2},P_{\epsilon,j}}(y)Z_{\epsilon,j}(y)\right)\Big(Y(\epsilon y) \Phi_{\epsilon}(\epsilon y)-Y(0)\Phi_{\epsilon}(0)\Big) dy}_{:=F_{21}}\\
&+\underbrace{\displaystyle{\int_{\mathbb{R}^N}}\left(\sum_{i,j=1}^{k}\sum_{i\neq j}U_{\mu_{2},P_{\epsilon,j}}(y)Z_{\epsilon,i}(y)\right)\Big(Y(\epsilon y) \Phi_{\epsilon}(\epsilon y)-Y(0)\Phi_{\epsilon}(0)\Big) dy}_{:=F_{22}},
\end{align*}
since
\begin{align*}
    F_{21}=&-k\displaystyle{\int_{\mathbb{R}^N}}U_{\mu_{2},P_{\epsilon,1}}(y)\,\frac{\partial }{\partial y_{1}}U_{\mu_{2},P_{\epsilon,1}}(y)\Big(Y(\epsilon y) \Phi_{\epsilon}(\epsilon y)-Y(0)\Phi_{\epsilon}(0)\Big) dy\\
    =&-\frac{k}{2}\displaystyle{\int_{\mathbb{R}^N}}\frac{\partial }{\partial y_{1}}U^{2}_{\mu_{2},P_{\epsilon,1}}(y)\Big(Y(\epsilon y) \Phi_{\epsilon}(\epsilon y)-Y(0)\Phi_{\epsilon}(\epsilon y)+Y(0)\Phi_{\epsilon}(\epsilon y)-Y(0)\Phi_{\epsilon}(0)\Big) dy\\
    =&\underbrace{-\frac{k}{2}\displaystyle{\int_{\mathbb{R}^N}}\frac{\partial }{\partial y_{1}}U^{2}_{\mu_{2},P_{\epsilon,1}}(y)\Big(Y(\epsilon y) -Y(0)\Big) \Phi_{\epsilon}(\epsilon y)dy}_{:=F_{211}}\\
    &\underbrace{-\frac{k}{2}\displaystyle{\int_{\mathbb{R}^N}}\frac{\partial }{\partial y_{1}}U^{2}_{\mu_{2},P_{\epsilon,1}}(y)\Big(\Phi_{\epsilon}(\epsilon y) -\Phi_{\epsilon}(0)\Big) Y(0)dy}_{:=F_{212}},
\end{align*}
where
\begin{align*}
    F_{211}\lesssim&\|\Phi_\epsilon\|_{L^{\infty}(\R^N)}\displaystyle{\int_{\mathbb{R}^N}}U^{2}_{\mu_{2},P_{\epsilon,1}}(y)\Big|Y(\epsilon y) -Y(0)\Big|dy\\\lesssim&\epsilon^{\frac{N}{2}}\displaystyle{\int_{\mathbb{R}^N}}U_{\mu_2}^{2}\left(y-\frac{P_{1}}{\epsilon}\right)|\epsilon y|^{2}dy\lesssim\epsilon^{\frac{N}{2}}\epsilon^{2}\left(\frac{\rho}{\epsilon}\right)^{2}=o(\epsilon\rho),
\end{align*}
here we used the fact that $\left|\nabla Y(0)\right|=0$. Moreover, using Remark \ref{after}, we have
\begin{align*}
F_{212}=&\frac{k}{2}\displaystyle{\int_{\mathbb{R}^N}}\frac{\partial }{\partial y_{1}}\Big(\Phi_{\epsilon}(\epsilon y) -\Phi_{\epsilon}(0)\Big)\,U^{2}_{\mu_{2},P_{\epsilon,1}}(y)Y(0)dy\\
    =&\frac{k}{2}\,Y(0)\displaystyle{\int_{\mathbb{R}^N}}\epsilon\,\left(\frac{\partial }{\partial y_{1}}\Phi_{\epsilon}(\epsilon y) -\frac{\partial }{\partial y_{1}}\Phi_{\epsilon}(0)\right)\,U^{2}_{\mu_{2},P_{\epsilon,1}}(y)dy\\
    \lesssim&\epsilon\,\left\|\Phi_\epsilon\right\|_{C^{1,1-\frac{N}{q}}(\R^N)}\,\|Y\|_{L^{\infty}(\R^N)}\displaystyle{\int_{\mathbb{R}^N}}U_{\mu_2}^{2}\left(y-\frac{P_{1}}{\epsilon}\right)\,|\epsilon y|^{1-\frac{N}{q}}dy\\
    \lesssim&\epsilon^{1+\frac{N}{q}}\rho^{1-\frac{N}{q}}=o(\epsilon\rho).
\end{align*}
Thus, we derive that
$$F_{21}=o(\epsilon\rho).$$
As for $F_{22}$, a direct calculation yields that
\begin{align*}
    F_{22}=&\underbrace{\displaystyle{\int_{\mathbb{R}^N}}\left(\sum_{i,j=1}^{k}\sum_{i\neq j}U_{\mu_{2},P_{\epsilon,j}}(y)Z_{\epsilon,i}(y)\right)\Big(Y(\epsilon y) -Y(0)\Big) \Phi_{\epsilon}(\epsilon y) dy}_{:=F_{221}}\\
    &+\underbrace{\displaystyle{\int_{\mathbb{R}^N}}\left(\sum_{i,j=1}^{k}\sum_{i\neq j}U_{\mu_{2},P_{\epsilon,j}}(y)Z_{\epsilon,i}(y)\right)\Big(\Phi_{\epsilon}(\epsilon y) -\Phi_{\epsilon}(0)\Big) Y(0)dy}_{:=F_{222}},
\end{align*}
where
\begin{align*}
    F_{221}\lesssim&\|\Phi_\epsilon\|_{L^{\infty}(\R^N)}\sum_{i,j=1}^{k}\sum_{i\neq j}\displaystyle{\int_{\mathbb{R}^N}}U_{\mu_{2},P_{\epsilon,j}}(y)\,U_{\mu_{2},P_{\epsilon,i}}(y)\Big|Y(\epsilon y) -Y(0)\Big|dy\\
    \lesssim&\|\Phi_\epsilon\|_{L^{\infty}(\R^N)}\|Y\|_{L^{\infty}(\R^N)}\left(\sum_{i=2}^{k}\displaystyle{\int_{\mathbb{R}^N}}U_{\mu_{2},P_{\epsilon,1}}(y)\,U_{\mu_{2},P_{\epsilon,2}}(y)dy\right)\\
    \lesssim&\epsilon^{\frac{N}{2}}\,e^{-2\sqrt{\omega_{0}}\frac{\rho}{\epsilon}\sin{\frac{\pi}{k}}}\left(\frac{\rho}{\epsilon}\right)^{\frac{3-N}{2}}=o(\epsilon\rho),
\end{align*}
and similarly,
\begin{align*}
F_{222}\lesssim o(\epsilon\rho),
\end{align*}
which implies that
$$F_{22}=o(\epsilon\rho).$$
Thus, we conclude that
\begin{align*}
    F_{2}=o(\epsilon\rho).
\end{align*}

Furthermore, by the estimates of correction terms $\Phi_\epsilon$ and $\Psi_\epsilon$ obtained in Lemma \ref{Phi}, a simple computation yields that
\begin{align*}
|F_{5}|\lesssim\|\Phi_\epsilon\|^{2}_{L^{\infty}(\R^N)}=O\left(\epsilon^{N}\right)=o\left(\epsilon \rho\right),\quad|F_{6}|\lesssim\|\Phi_\epsilon\|_{L^{\infty}(\R^N)}\|\Psi_\epsilon\|_{L^{2}(\R^N)}=O\left(\epsilon^{N}\right)=o\left(\epsilon \rho\right),
\end{align*}
and
\begin{align*} |F_{7}|\lesssim\|\Psi_\epsilon\|^{2}_{L^{\infty}(\R^N)}=O\left(\epsilon^{N}\right)=o\left(\epsilon \rho\right),\quad|F_{8}|\lesssim\|\Psi_\epsilon\|^{3}_{H^{2}(\R^N)}=O\left(\epsilon^{\frac{3N}{2}}\right)=o\left(\epsilon \rho\right).
\end{align*}

In the following, we study the cubic and square terms $F_3$, $F_4$ and  $F_9$. Firstly, we calculate the term $F_9$, which is an present only { if $\alpha\neq0$}.
Indeed, a simple computation yields that
\begin{align*}
    {F_{9}}=&{\underbrace{\alpha\sum\limits_{i=2}^{m}\displaystyle{\int_{\mathbb{R}^N}}U_{\rho,\epsilon}U^{2}_{\rho,\epsilon}\left(\widehat{\mathcal{R}_{i}}    y\right)Z_{\epsilon} dy}_{:=F_{91}}}\\
    &+\underbrace{2\alpha\beta\sum\limits_{i=2}^{m}\displaystyle{\int_{\mathbb{R}^N}}U_{\rho,\epsilon}\Psi_{\epsilon}\left(\widehat{\mathcal{R}_{i}}    y\right)U_{\rho,\epsilon}\left(\widehat{\mathcal{R}_{i}}    y\right)Z_{\epsilon} dy}_{:= {F_{92}}}\\
    &+\underbrace{\alpha\beta\sum\limits_{i=2}^{m}\displaystyle{\int_{\mathbb{R}^N}}\Psi_{\epsilon}\left( y\right)U^{2}_{\rho,\epsilon}\left(\widehat{\mathcal{R}_{i}}    y\right)Z_{\epsilon} dy}_{:= {F_{93}}}+h.o.r.\quad,
\end{align*}
here the $h.o.r$ terms include the correction term $\Psi_{\epsilon}^{l}$, $l\geq 2$, then we can omit them as \eqref{es-4}. As for $F_{91}$, a simple computation yields that
\begin{align}\label{F-91}
    {F_{91}}=&\alpha\sum\limits_{j=1}^{k}\displaystyle{\int_{\mathbb{R}^N}}U_{\mu_{2},P_{\epsilon,j}}(y)\,U^{2}_{\mu_{2},P_{\epsilon,j}}\left(\widehat{\mathcal{R}_{2}}    y\right)\,Z_{\epsilon,j}(y) dy\nonumber\\=&\alpha k\displaystyle{\int_{\mathbb{R}^N}}U_{\mu_{2},P_{\epsilon,1}}\,U^{2}_{\mu_{2},P_{\epsilon,12}}\,Z_{\epsilon,1} dy+h.o.t.,
\end{align}
where
\begin{align*}
h.o.t.:=&\alpha\sum\limits_{i=3}^{m}\displaystyle{\int_{\mathbb{R}^N}}U_{\rho,\epsilon}(y)U^{2}_{\rho,\epsilon}\left(\widehat{\mathcal{R}_{i}}    y\right)Z_{\epsilon} dy\\&+\alpha\sum\limits_{q,l=1}^{k}\sum_{q\neq l}\displaystyle{\int_{\mathbb{R}^N}}U_{\rho,\epsilon}(y)\,U_{\mu_{2},P_{\epsilon,q}}\left(\widehat{\mathcal{R}_{2}}    y\right)U_{\mu_{2},P_{\epsilon,l}}\left(\widehat{\mathcal{R}_{2}}    y\right)Z_{\epsilon}(y)dy\\&+\alpha\sum\limits_{j,q=1}^{k}\sum_{q\neq j}\displaystyle{\int_{\mathbb{R}^N}}U_{\mu_{2},P_{\epsilon,j}}(y)\,U^{2}_{\mu_{2},P_{\epsilon,q}}\left(\widehat{\mathcal{R}_{2}}    y\right)Z_{\epsilon}(y) dy\\
    &+\alpha\sum\limits_{l,j=1}^{k}\sum_{l\neq j}\displaystyle{\int_{\mathbb{R}^N}}U_{\mu_{2},P_{\epsilon,j}}(y)\,U^{2}_{\mu_{2},P_{\epsilon,j}}\left(\widehat{\mathcal{R}_{2}}    y\right)\,Z_{\epsilon,l}(y) dy.
\end{align*}
Then combining \eqref{F-91} with Lemma \ref{derivative}, we conclude that
\begin{align*}
    {F_{91}}=&\alpha k\displaystyle{\int_{\mathbb{R}^N}}U_{\mu_{2}}(y)\,U^{2}_{\mu_{2}}\left(y-\frac{P_{12}-P_1}{\epsilon}\right)\,\left(-\frac{\partial}{\partial y_1}U_{\mu_{2}}(y)\right)dy+h.o.r,\\&=-\frac{\alpha k}{2}\displaystyle{\int_{\mathbb{R}^N}}{U^{2}_{\mu_{2}}\left(y-\frac{P_{12}-P_1}{\epsilon}\right)\,\frac{\partial}{\partial y_1}U^{2}_{\mu_{2}}(y)}dy+h.o.r\\
    = &\begin{cases}
-\alpha\,B_2\,e^{-4\sqrt{\omega_{0}}\frac{\rho}{\epsilon}\sin{\frac{\pi}{mk}}}\left(\frac{\rho}{\epsilon}\right)^{-\frac{1}{2}}\left(1+o(1)\right),\quad N=2,\\
-\alpha B_2\,e^{-4\sqrt{\omega_{0}}\frac{\rho}{\epsilon}\sin{\frac{\pi}{mk}}}\left(\frac{\rho}{\epsilon}\right)^{-{2}}\ln{\left(\frac{\rho}{\epsilon}\right)}\left(1+o(1)\right),\quad N=3,
        \end{cases}
\end{align*}
for some positive constant  $B_2.$
Moreover,
\begin{align*}
    |F_{92}|\lesssim&\displaystyle{\int_{\mathbb{R}^N}}U_{\rho,\epsilon}\Psi_{\epsilon}\left(\widehat{\mathcal{R}_{2}}    y\right)U_{\rho,\epsilon}\left(\widehat{\mathcal{R}_{2}}    y\right)U_{\rho,\epsilon} dy\\
   \lesssim&\underbrace{\left(\displaystyle{\int_{\mathbb{R}^N}}U^{2}_{\rho,\epsilon}(y)\Psi^{2}_{\epsilon}\left(\widehat{\mathcal{R}_{2}}    y\right) \right)^{\frac{1}{2}}}_{:=F_{921}}\underbrace{\left(\displaystyle{\int_{\mathbb{R}^N}}U^{2}_{\rho,\epsilon}(y)U^{2}_{\rho,\epsilon}\left(\widehat{\mathcal{R}_{2}}    y\right) \right)^{\frac{1}{2}}}_{:=F_{922}},
\end{align*}
then using similar arguments as in the estimates of $M_{61}$-$M_{63}$, we have
\begin{align*}
    F_{921}\lesssim&\left(\displaystyle\int_{\R^N}4\beta^{2}\Phi_{\epsilon}^{2}(0)Y^{2}(0)\left(\sum_{j=1}^{k}\Psi^{2}\left(y-\frac{P_{j2}}{\epsilon}\right)\right)\,\left(\sum_{h=1}^{k}U_{\mu_{2},P_{\epsilon,h}}^{2}\left(y\right)\right)dy\right)^{\frac{1}{2}}\\
    \lesssim&\|\Phi_{\epsilon}\|_{L^{\infty}(\R^N)}\left(\displaystyle\int_{\R^N}\Psi^{2}\left(y-\frac{P_{12}}{\epsilon}\right)\,U_{\mu_{2}}^{2}\left(y-\frac{P_{1}}{\epsilon}\right)dy\right)^{\frac{1}{2}}\\
    \lesssim&\epsilon^{\frac{N}{2}}\left(\displaystyle\int_{\R^N}\Psi^{2}\left(y\right)\,U_{\mu_{2}}^{2}\left(y-\frac{P_1-P_{12}}{\epsilon}\right)dy\right)^{\frac{1}{2}}\\
    =&\epsilon^{\frac{N}{2}}\left(\displaystyle\int_{B_{\frac{|P_{12}-P_1|}{2\epsilon}}}\Psi^{2}\left(y\right)\,U_{\mu_{2}}^{2}\left(y-\frac{P_1-P_{12}}{\epsilon}\right)+\displaystyle\int_{B^{c}_{\frac{|P_{12}-P_1|}{2\epsilon}}}\Psi^{2}\left(y\right)\,U_{\mu_{2}}^{2}\left(y-\frac{P_1-P_{12}}{\epsilon}\right)\right)^{\frac{1}{2}}\\
    \lesssim&\epsilon^{\frac{N}{2}}\left(e^{-\sqrt{\omega_0}\,\frac{|P_{12}-P_1|}{\epsilon}}\displaystyle\int_{\R^N}\Psi^{2}+e^{-\gamma\,\frac{|P_{12}-P_1|}{\epsilon}}\displaystyle\int_{\R^N}U_{\mu_{2}}^{2}\right)^{\frac{1}{2}}\\\lesssim&\epsilon^{\frac{N}{2}}e^{-\frac{\gamma}{2}\,\frac{|P_{12}-P_1|}{\epsilon}},
\end{align*}
and
\begin{align*}
    F_{922}\lesssim&\left(\displaystyle\int_{\R^N}U^{2}_{\mu_{2},P_{\epsilon,1}}(y)\,U_{\mu_{2},P_{\epsilon,12}}^{2}\left( y\right)dy\right)^{\frac{1}{2}}\\
    \lesssim&\begin{cases}
        e^{-{\sqrt{\omega_0}}\,\frac{|P_{12}-P_1|}{\epsilon}}\left(\frac{\rho}{\epsilon}\right)^{-\frac{1}{4}},\quad &N=2,\\
         e^{-{\sqrt{\omega_0}}\,\frac{|P_{12}-P_1|}{\epsilon}}\left(\frac{\rho}{\epsilon}\right)^{-1}\left(\ln{\frac{\rho}{\epsilon}}\right)^{\frac{1}{2}},\quad &N=3.\\
    \end{cases}
\end{align*}
Thus,
\begin{align*}
    {|F_{92}|}   \lesssim&\begin{cases}
       \epsilon\,e^{-2\left(\frac{\gamma}{2}+\sqrt{\omega_0}\right)\frac{\rho}{\epsilon}\sin{\frac{\pi}{mk}}} \left(\frac{\rho}{\epsilon}\right)^{-\frac{1}{4}},\quad &N=2,\\
\epsilon^{\frac{3}{2}}\,e^{-2\left(\frac{\gamma}{2}+\sqrt{\omega_0}\right)\frac{\rho}{\epsilon}\sin{\frac{\pi}{mk}}}\left(\frac{\rho}{\epsilon}\right)^{-1}\left(\ln{\frac{\rho}{\epsilon}}\right)^{\frac{1}{2}},\quad &N=3.\\
    \end{cases}\\
    =&\begin{cases}
      o\left(e^{^{-4\sqrt{\omega_0}\frac{\rho}{\epsilon}\sin{\frac{\pi}{mk}}}}\left(\frac{\rho}{\epsilon}\right)^{-\frac{1}{2}}\right),\quad N=2,\\
      o\left(e^{^{-4\sqrt{\omega_0}\frac{\rho}{\epsilon}\sin{\frac{\pi}{mk}}}}\left(\frac{\rho}{\epsilon}\right)^{-{2}}\ln{\left(\frac{\rho}{\epsilon}\right)}\right),\quad N=3.
  \end{cases}
\end{align*}
Similarly,
\begin{align*}
    |F_{93}|   \lesssim&\displaystyle{\int_{\mathbb{R}^N}}U_{\rho,\epsilon}(y)\Psi_{\epsilon}\left(y\right)U^{2}_{\rho,\epsilon}\left(\widehat{\mathcal{R}_{2}}    y\right) dy\\\lesssim&\left(\displaystyle{\int_{\mathbb{R}^N}}U^{2}_{\rho,\epsilon}\left(\widehat{\mathcal{R}_{2}}    y\right)\Psi^{2}_{\epsilon}\left( y\right) \right)^{\frac{1}{2}}\left(\displaystyle{\int_{\mathbb{R}^N}}U^{2}_{\rho,\epsilon}(y)U^{2}_{\rho,\epsilon}\left(\widehat{\mathcal{R}_{2}}    y\right) \right)^{\frac{1}{2}}\\=&\begin{cases}
      o\left(e^{^{-4\sqrt{\omega_0}\frac{\rho}{\epsilon}\sin{\frac{\pi}{mk}}}}\left(\frac{\rho}{\epsilon}\right)^{-\frac{1}{2}}\right),\quad N=2,\\
      o\left(e^{^{-4\sqrt{\omega_0}\frac{\rho}{\epsilon}\sin{\frac{\pi}{mk}}}}\left(\frac{\rho}{\epsilon}\right)^{-{2}}\ln{\left(\frac{\rho}{\epsilon}\right)}\right),\quad N=3.
  \end{cases}
\end{align*}
Thus, from the estimates of $F_{91}$-$F_{93}$, we derive that
\begin{equation}\label{nine}
    {F_{9}}
    =\begin{cases}
-\alpha\,B_2\,e^{-4\sqrt{\omega_{0}}\frac{\rho}{\epsilon}\sin{\frac{\pi}{mk}}}\left(\frac{\rho}{\epsilon}\right)^{-\frac{1}{2}}\left(1+o(1)\right),\quad N=2,\\
-\alpha B_2\,e^{-4\sqrt{\omega_{0}}\frac{\rho}{\epsilon}\sin{\frac{\pi}{mk}}}\left(\frac{\rho}{\epsilon}\right)^{-{2}}\ln{\left(\frac{\rho}{\epsilon}\right)}\left(1+o(1)\right),\quad N=3,
        \end{cases}
\end{equation}
for some positive constant $B_2$.

{Next, we calculate the term  $F_3$}. A direct computation yields that
\begin{align*}
    F_3=&\underbrace{6\mu_2 \sum\limits_{i,j=1}^k\sum\limits_{i>j}\displaystyle{\int_{\mathbb{R}^N}}\left(U^{{2}}_{\mu_{2},P_{\epsilon,j}}U_{\mu_{2},P_{\epsilon,i}}\right) Z_{{\epsilon,j}}dy}_{:=F_{31}}\\
    &+\underbrace{6\mu_2 \sum\limits_{i,j,l=1}^k\sum\limits_{i>j}\sum\limits_{l\neq j}\displaystyle{\int_{\mathbb{R}^N}}\left(U^{{2}}_{\mu_{2},P_{\epsilon,j}}U_{\mu_{2},P_{\epsilon,i}}\right) Z_{{\epsilon,l}} dy}_{:=F_{32}}\\
    &+\underbrace{O\left(\sum\limits_{i,j,l=1}^k\sum\limits_{i\neq j\neq l}\displaystyle{\int_{\mathbb{R}^N}}\left(U_{\mu_{2},P_{\epsilon,j}}U_{\mu_{2},P_{\epsilon,i}}U_{\mu_{2},P_{\epsilon,l}}\right) Z_\epsilon dy\right)}_{:=F_{33}}.
    \end{align*}
As for $F_{31}$, {when $\alpha=0$ $(m=1)$}, we use Lemma \ref{derivative} to deduce that
\begin{align*}
    F_{31}=&6k\mu_2 \sum\limits_{i=2}^k\displaystyle{\int_{\mathbb{R}^N}}U^{{2}}_{\mu_{2},P_{\epsilon,1}}\,U_{\mu_{2},P_{\epsilon,i}}\, Z_{{\epsilon,1}} dy\\
    =&6k\mu_2 \sum\limits_{i=2}^k\displaystyle{\int_{\mathbb{R}^N}}U^{{2}}_{\mu_{2}}(x)\,U_{\mu_{2}}\left(x-\frac{P_{i}-P_{1}}{\epsilon}\right)\, \left(-\frac{\partial}{\partial x_1}U_{\mu_{2}}(x)\right) dx\\
    =&-2k\mu_2 \sum\limits_{i=2}^k\displaystyle{\int_{\mathbb{R}^N}}\,U_{\mu_{2}}\left(x-\frac{P_{i}-P_{1}}{\epsilon}\right)\, \left(\frac{\partial}{\partial x_1}U^{3}_{\mu_{2}}(x)\right) dx\\
    =&-B_{1}e^{-2\sqrt{\omega_{0}}\frac{\rho}{\epsilon}\sin{\frac{\pi}{k}}}\left(\frac{\rho}{\epsilon}\right)^{-\frac{N-1}{2}}\left(1+o(1)\right),
\end{align*}
for some positive constant $B_1$. Moreover, using Lemma \ref{lma2.2}, we infer that
\begin{align*}
    |F_{32}|\lesssim&\sum\limits_{i=2}^k\displaystyle{\int_{\mathbb{R}^N}}U^{{2}}_{\mu_{2},P_{\epsilon,1}}U^{2}_{\mu_{2},P_{\epsilon,i}} dy=\sum\limits_{i=2}^k\displaystyle{\int_{\mathbb{R}^N}}U^{{2}}_{\mu_{2}}(x)\,U^{2}_{\mu_{2}}\left(x-\frac{P_{i}-P_{1}}{\epsilon}\right)dx\\
    \lesssim&
    \begin{cases}
e^{-4\sqrt{\omega_{0}}\frac{\rho}{\epsilon}\sin{\frac{\pi}{k}}}\left(\frac{\rho}{\epsilon}\right)^{-\frac{1}{2}},\quad&\text{$N=2$,}\\
e^{-4\sqrt{\omega_{0}}\frac{\rho}{\epsilon}\sin{\frac{\pi}{k}}}\left(\frac{\rho}{\epsilon}\right)^{-2}\ln{\left|\frac{\rho}{\epsilon}\right|},\quad&\text{$N=3$}\end{cases}
\\=&o\left(e^{-2\sqrt{\omega_{0}}\frac{\rho}{\epsilon}\sin{\frac{\pi}{k}}}\left(\frac{\rho}{\epsilon}\right)^{-\frac{N-1}{2}}\right),
\end{align*}
similarly, we have
$$|F_{33}|=o\left(e^{-2\sqrt{\omega_{0}}\frac{\rho}{\epsilon}\sin{\frac{\pi}{k}}}\left(\frac{\rho}{\epsilon}\right)^{-\frac{N-1}{2}}\right).$$
Therefore, we conclude that if $\alpha=0$,
$$
F_{3}=-B_{1}e^{-2\sqrt{\omega_{0}}\frac{\rho}{\epsilon}\sin{\frac{\pi}{k}}}\left(\frac{\rho}{\epsilon}\right)^{-\frac{N-1}{2}}\left(1+o(1)\right),
$$
for some positive constant $B_{1}$.
We observe that if $\alpha=0$, $F_3$ is the leading term of $\displaystyle{\int_{\mathbb{R}^N}} \mathcal{E}_2 Z_{\epsilon} dy$.
However, when $\alpha\neq0$ $(m\geq 2)$, it is negligible with respect to  $F_9$ obtained in \eqref{nine}, which turns out to be the leading term.
Indeed,
\begin{align*}
   \left|F_{3}\right|=& \left|B_{1}\right|e^{-2\sqrt{\omega_{0}}\frac{\rho}{\epsilon}\sin{\frac{\pi}{k}}}\left(\frac{\rho}{\epsilon}\right)^{-\frac{N-1}{2}}\left(1+o(1)\right)\\
  {=}&\begin{cases}
      o\left(e^{^{-4\sqrt{\omega_0}\frac{\rho}{\epsilon}\sin{\frac{\pi}{mk}}}}\left(\frac{\rho}{\epsilon}\right)^{-\frac{1}{2}}\right),\quad N=2,\\
      o\left(e^{^{-4\sqrt{\omega_0}\frac{\rho}{\epsilon}\sin{\frac{\pi}{mk}}}}\left(\frac{\rho}{\epsilon}\right)^{-{2}}\ln{\left(\frac{\rho}{\epsilon}\right)}\right),\quad N=3,
  \end{cases}
\end{align*}
because
$
\sin{\frac{\pi}{k}}>2\sin{\frac{\pi}{mk}} \ \hbox{if}\ m\ge2.
$
Now, a simple computation yields that
\begin{align*}
|F_4|\lesssim&\underbrace{\sum\limits_{i=2}^k\displaystyle{\int_{\mathbb{R}^N}}U^{{2}}_{\mu_{2},P_{\epsilon,1}}U_{\mu_{2},P_{\epsilon,i}}\,\Psi_{\epsilon}\, dy}_{:=F_{41}}+\underbrace{\sum\limits_{i=2}^k\displaystyle{\int_{\mathbb{R}^N}}U_{\mu_{2},P_{\epsilon,1}}U^{{2}}_{\mu_{2},P_{\epsilon,i}}\,\Psi_{\epsilon}\, dy}_{:=F_{42}}\\&+\underbrace{O\left(k\sum\limits_{l,i=2}^k\sum_{l\neq i}\displaystyle{\int_{\mathbb{R}^N}}U_{\mu_{2},P_{\epsilon,1}}U_{\mu_{2},P_{\epsilon,i}}\,U_{\mu_{2},P_{\epsilon,l}}\,\Psi_{\epsilon}\,dy\right)}_{:=F_{43}}.
\end{align*}
If {$\alpha=0$}, by the estimate of $\Psi_\epsilon$ stated in \eqref{es-4} and Lemma \ref{lma2.2}, we have
\begin{align*}
    |F_{41}|\lesssim&\|\Psi_\epsilon\|_{L^{2}(\R^N)}\sum\limits_{i=2}^k\left(\displaystyle{\int_{\R^N}}U^{4}_{\mu_{2},P_{\epsilon,1}}(y)\,U^{2}_{\mu_{2},P_{\epsilon,i}}(y)\,dy\right)^{\frac{1}{2}}\\\lesssim&
    \epsilon^{\frac{N}{2}}\,e^{-2\sqrt{\omega_{0}}\frac{\rho}{\epsilon}\sin{\frac{\pi}{k}}}\left(\frac{\rho}{\epsilon}\right)^{-\frac{N-1}{2}}=o\left(e^{-2\sqrt{\omega_{0}}\frac{\rho}{\epsilon}\sin{\frac{\pi}{k}}}\left(\frac{\rho}{\epsilon}\right)^{-\frac{N-1}{2}}\right),
\end{align*}
similarly, we get that
\begin{align*}
    |F_{42}|=o\left(e^{-2\sqrt{\omega_{0}}\frac{\rho}{\epsilon}\sin{\frac{\pi}{k}}}\left(\frac{\rho}{\epsilon}\right)^{-\frac{N-1}{2}}\right),\quad |F_{43}|=o\left(e^{-2\sqrt{\omega_{0}}\frac{\rho}{\epsilon}\sin{\frac{\pi}{k}}}\left(\frac{\rho}{\epsilon}\right)^{-\frac{N-1}{2}}\right),
\end{align*}
which implies that
$$|F_{4}|=o\left(e^{-2\sqrt{\omega_{0}}\frac{\rho}{\epsilon}\sin{\frac{\pi}{k}}}\left(\frac{\rho}{\epsilon}\right)^{-\frac{N-1}{2}}\right).$$
If {$\alpha\neq0$}, combining the above estimates with \eqref{bala2}, we have
\begin{align*}
|F_4|=\begin{cases}
      o\left(e^{^{-4\sqrt{\omega_0}\frac{\rho}{\epsilon}\sin{\frac{\pi}{mk}}}}\left(\frac{\rho}{\epsilon}\right)^{-\frac{1}{2}}\right),\quad N=2,\\
      o\left(e^{^{-4\sqrt{\omega_0}\frac{\rho}{\epsilon}\sin{\frac{\pi}{mk}}}}\left(\frac{\rho}{\epsilon}\right)^{-{2}}\ln{\left(\frac{\rho}{\epsilon}\right)}\right),\quad N=3.
  \end{cases}
\end{align*}

Thus, {from all the estimates of $F_{1}$-$F_{9}$}, we get \eqref{c-alpha} and \eqref{conclusion}, respectively.
\\

{It remains to estimate the higher order terms of the R.H.S of \eqref{ceps}}. By the estimates of $(\phi_\epsilon, \psi_\epsilon)$ obtained in Lemma \ref{contraction}, we deduce that either $\alpha=0$ or $\alpha\neq0$,
\begin{align}\label{N2}
\displaystyle{\int_{\mathbb{R}^N} }\mathcal{N}_2(\phi_\epsilon, \psi_\epsilon) Z_{\epsilon} dy \lesssim o\left(\epsilon\rho\right).
\end{align}

Moreover, from \eqref{z-ep} and \eqref{linear-op}, we get that
\begin{align}\label{L-2}
&\displaystyle{\int_{\mathbb{R}^N}} \mathcal{L}_2(\phi_\epsilon, \psi_\epsilon) Z_{\epsilon}dy\nonumber\\
=& \underbrace{\displaystyle{\int_{\mathbb{R}^N}}\Big(\omega(\epsilon y)-\omega_0\Big) Z_{\epsilon} \psi_\epsilon \,dy}_{:=P_1}-3 \mu_2 \underbrace{\displaystyle{\int_{\mathbb{R}^N}} \left(\sum\limits_{i,j=1}^k\sum_{i\neq j} U_{\mu_{2},P_{\epsilon,j}}^{2}(y)Z_{{\epsilon,i}}(y)\right)  \psi_\epsilon
dy}_{:=P_2}\nonumber\\& -3 \mu_2 \underbrace{\displaystyle{\int_{\mathbb{R}^N}} \left(\sum\limits_{i,j=1}^k\sum_{i\neq j} U_{\mu_{2},P_{\epsilon,j}}(y)U_{\mu_{2},P_{\epsilon,i}}(y)\right)  Z_{\epsilon}\psi_\epsilon dy}_{:=P_3}\underbrace{-3 \mu_2 \beta^2 \displaystyle{\int_{\mathbb{R}^N}} \Psi_{\epsilon}^2 Z_{\varepsilon} \psi_\epsilon dy}_{:=P_4}\nonumber\\&\underbrace{-6 \mu_2 \beta \displaystyle{\int_{\mathbb{R}^N}}\left(\sum\limits_{j=1}^kU_{\mu_{2},P_{\epsilon,j}}(y)\right)\Psi_{\epsilon}(y) Z_{\epsilon}(y) \psi_\epsilon (y)dy}_{{:=P_5}}\underbrace{-\beta^3 \displaystyle{\int_{\mathbb{R}^N}} \Phi_{\epsilon}^2(\epsilon y) Z_{\epsilon} \psi_\epsilon dy}_{{:=P_6}}\nonumber\\&\underbrace{-2 \beta^2 \displaystyle{\int_{\mathbb{R}^N}} Y(\epsilon y) \Phi_{\epsilon}(\epsilon y) Z_{\epsilon} (y)\psi_\epsilon (y)dy}_{{:=P_7}}\underbrace{-2 \beta^2 \displaystyle{\int_{\mathbb{R}^N}} \Psi_{\epsilon}(y) Y(\epsilon y) \phi_\epsilon(\epsilon y) Z_{\epsilon}(y) dy}_{:=P_{8}}\nonumber \\
&\underbrace{-2 \beta^3 \displaystyle{\int_{\mathbb{R}^N}}\Psi_{\epsilon} \Phi_{\epsilon}(\epsilon y) \phi_{\epsilon}(\epsilon y) Z_{\epsilon} dy}_{:=P_{9}}\underbrace{-2 \beta \displaystyle{\int_{\mathbb{R}^N}}U_{\rho,\epsilon} Y(\epsilon y) \phi_\epsilon
(\epsilon y) Z_{\epsilon} dy}_{:=P_{10}}\nonumber\\
&\underbrace{-2 \beta^2 \displaystyle{\int_{\mathbb{R}^N}}U_{\rho,\epsilon} \Phi_{\epsilon}(\epsilon y) \phi_{\epsilon}(\epsilon y) Z_{\epsilon} dy}_{:=P_{11}}\nonumber\\
&{\underbrace{-2\alpha\displaystyle{\int_{\mathbb{R}^N}}\Theta_{\epsilon}\left(\sum_{i=2}^{m}\Theta_{\epsilon}\left(\widehat{\mathcal{R}}_{i}   {y}{}\right)  \psi_\epsilon\left(\widehat{\mathcal{R}}_{i}   {y}{}\right)\right)Z_{\epsilon} dy}_{:=P_{12}}\underbrace{-\alpha\displaystyle\int_{\R^N}\psi_\epsilon\left(\sum_{i=2}^{m}\Theta^{2}_{\epsilon}\left(\widehat{\mathcal{R}}_{i}   {y}{}\right) \right)Z_{\epsilon} dy}_{:=P_{13}}},
\end{align}
{we will estimate each term in the right hand side of \eqref{L-2} to  obtain that
\begin{align}
\displaystyle\int_{\mathbb{R}^N} \mathcal{L}_2(\phi_\epsilon, \psi_\epsilon) Z_{\epsilon} dy\lesssim&\,\, o\left(\epsilon\rho\right),\label{hope}
    \end{align}
provided that $\alpha=0$ and $\alpha\neq 0$, respectively}. Indeed, {by the estimates of $(\phi_\epsilon, \psi_\epsilon)$ obtained in Lemma \ref{contraction}},
\begin{align*}
    |P_1|&\lesssim\rho^{2}{\|\psi_\epsilon\|_{L^{2}(\R^N)}}\lesssim\rho^{2}{\|(\phi_\epsilon, \psi_\epsilon)\|}=o\left(\epsilon\rho\right).
\end{align*}
From \eqref{bala1} and \eqref{bala2}, we derive that both $\alpha=0$ and $\alpha\neq 0$
\begin{align*}
    |P_2|&\lesssim{\|(\phi_\epsilon, \psi_\epsilon)\|}\,e^{-2\sqrt{\omega_{0}}\frac{\rho}{\epsilon}\sin{\frac{\pi}{k}}}\left(\frac{\rho}{\epsilon}\right)^{-\frac{N-1}{2}}\lesssim o\left(\epsilon\rho\right).
\end{align*}
Similarly,
\begin{align*}
    |P_3|\lesssim o\left(\epsilon\rho\right).
\end{align*}
Moreover, by the estimates of correction terms $\Phi_\epsilon$ and $\Psi_\epsilon
    $ obtained in Lemma \ref{Phi}, a simple computation yields that
\begin{align*}
|P_4|&\lesssim\|\psi_\epsilon\|_{L^{2}(\R^N)}\|\Psi_\epsilon\|^{2}_{L^{\infty}(\R^N)}\lesssim \epsilon^N{\|(\phi_\epsilon, \psi_\epsilon)\|}=o\left(\epsilon\rho\right),\\
{{|P_5|}}  &\lesssim \|\psi_\epsilon\|_{L^{2}(\R^N)}{\|\Psi_\epsilon\|_{L^{\infty}(\R^N)}}\lesssim {\epsilon^\frac{N}{2}}{\|(\phi_\epsilon, \psi_\epsilon)\|}=o\left(\epsilon\rho\right),\\
 |P_6|&\lesssim\|\psi_\epsilon\|_{L^{2}(\R^N)}\|\Phi_\epsilon\|^{2}_{L^{\infty}(\R^N)}\lesssim \epsilon^N{\|(\phi_\epsilon, \psi_\epsilon)\|}=o\left(\epsilon\rho\right),\\
{ {|P_7|}}  &\lesssim \|\psi_\epsilon\|_{L^{2}(\R^N)}{\|\Phi_\epsilon\|_{L^{\infty}(\R^N)}}\lesssim {\epsilon^\frac{N}{2}}{\|(\phi_\epsilon, \psi_\epsilon)\|}=o\left(\epsilon\rho\right),\\
{{|P_8|}}  &\lesssim \|\phi_\epsilon\|_{L^{\infty}(\R^N)}{\|\Psi_\epsilon\|_{H^{1}(\R^N)}}\lesssim {\epsilon^\frac{N}{2}}{\|(\phi_\epsilon, \psi_\epsilon)\|}=o\left(\epsilon\rho\right),\\|P_9|&\lesssim\|\phi_\epsilon\|_{L^{\infty}(\R^N)}\|\Phi_\epsilon\|_{L^{\infty}(\R^N)}\|\Psi_\epsilon\|_{H^{1}(\R^N)}\lesssim \epsilon^N{\|(\phi_\epsilon, \psi_\epsilon)\|}=o\left(\epsilon\rho\right).
\end{align*}

{As for ${P_{10}}$, there are some differences whether $\alpha$ equals zero or not}. Indeed, {if $\alpha=0$}, we have
\begin{align*}
P_{10}=&\underbrace{\displaystyle{\int_{\R^N}}U_{\rho,\epsilon}\, Y(\epsilon y) \Big(\phi_\epsilon
(\epsilon y)-\phi_\epsilon
(0) \Big)Z_{\epsilon} dy}_{:=D_1}\\
&+\underbrace{\displaystyle{\int_{\R^N}}U_{\rho,\epsilon}\,\phi_\epsilon
(0)  \Big(Y
(\epsilon y)-Y
(0) \Big)Z_{\epsilon} dy}_{:=D_2}+\underbrace{\displaystyle{\int_{\R^N}}U_{\rho,\epsilon}\,Y(0)\,\phi_\epsilon(0)Z_{\epsilon} dy}_{:=D_3}.
\end{align*}
Then from \eqref{al-zero} and the Sobolev embedding $H^{2}\left(\R^N\right)\hookrightarrow C^{0,\frac{1}{2}}\left(\R^N\right)$, $N=2,3$, we have
\begin{align*}
|D_1|
    &\lesssim\epsilon^{\frac{5}{2}}|\ln \epsilon|^2\,\|Z_\epsilon\|_{L^{2}(\R^N)}\left(\displaystyle{\int_{\R^N}}|y|\left(\sum\limits_{j=1}^kU^{{2}}_{\mu_{2},P_{\epsilon,j}}(y)\right)dy\right)^{\frac{1}{2}}\\&\lesssim\epsilon^{\frac{5}{2}}|\ln \epsilon|^2\left(\frac{\rho}{\epsilon}\right)^{\frac{1}{2}}=o\left(\epsilon\rho\right).
\end{align*}
Moreover, since $y=0$ is the critical point of $Y$, we derive that
\begin{align*}
|D_2|&\lesssim\epsilon^{{2}}\,\|\phi_\epsilon\|_{L^{\infty}(\R^N)}\,\|Z_\epsilon\|_{L^{2}(\R^N)}\left(\displaystyle{\int_{\R^3}}|y|^{4}\left(\sum\limits_{j=1}^kU^{{2}}_{\mu_{2},P_{\epsilon,j}}(y)\right)dy\right)^{\frac{1}{2}}\\&\lesssim\epsilon^{{4}}|\ln \epsilon|^2\left(\frac{\rho}{\epsilon}\right)^{{2}}=o\left(\epsilon\rho\right).
\end{align*}
In addition, by \eqref{z-1}, a simple computation yields that
\begin{align*}
    D_3&=Y(0)\,\phi_\epsilon(0)\displaystyle{\int_{\R^N}}\left(\sum\limits_{i,j=1}^k\sum_{i\neq j} U_{\mu_{2},P_{\epsilon,j}}(y)Z_{{\epsilon,i}}(y)\right)  dy\\
    &\lesssim\|Y\|_{L^{\infty}(\R^N)}\,\|\phi_\epsilon\|_{L^{\infty}(\R^N)}\left(\sum\limits_{i=2}^k\displaystyle{\int_{\R^N}}U_{\mu_{2},P_{\epsilon,1}}(y)\,U_{\mu_{2},P_{\epsilon,i}}(y)  dy\right)\\ &\lesssim\epsilon^{{2}}|\ln \epsilon|^2\,e^{-2\sqrt{\omega_{0}}\frac{\rho}{\epsilon}\sin{\frac{\pi}{k}}}\left(\frac{\rho}{\epsilon}\right)^{\frac{3-N}{2}}=o\left(\epsilon\rho\right).
\end{align*}
Thus, from the estimates of $D_1$-$D_3$, we conclude that when $\alpha=0$,
$$|P_{10}|=o\left(\epsilon\rho\right).$$
{However, when $\alpha\neq0$, we can not use the above arguments directly, and we need to improve the regularity of $\phi_\epsilon$, see Lemma \ref{imp} in detail}. Indeed, a direct computation yields that
\begin{align*}
    P_{10}=&\underbrace{\displaystyle{\int_{\mathbb{R}^N}} \left(\sum\limits_{j=1}^k U_{\mu_{2},P_{\epsilon,j}}(y)Z_{{\epsilon,j}}(y)\right)  \phi_\epsilon(\epsilon y)Y(\epsilon y)dy}_{:=D_4}\\&+\underbrace{\displaystyle{\int_{\mathbb{R}^N}} \left(\sum\limits_{i,j=1}^k \sum_{i\neq j}U_{\mu_{2},P_{\epsilon,j}}(y)Z_{{\epsilon,i}}(y)\right)  \phi_\epsilon(\epsilon y)Y(\epsilon y)
dy}_{:=D_5}.
\end{align*}
Firstly, using Lemma \ref{imp}, Lemma \ref{contraction} and the facts that
$\nabla Y(0)=0$
and $$
       \left|Y(y)-Y(0)\right|\lesssim\left\|Y\right\|_{C^{2}(\R^N)}|y|^{2}\lesssim|y|^{2}
,$$ a direct computation yields that
\begin{align*}
    D_4=&k\displaystyle{\int_{\mathbb{R}^N}} U_{\mu_{2},P_{\epsilon,1}}(y)Z_{{\epsilon,1}}(y) \phi_\epsilon(\epsilon y)Y(\epsilon y)dy\\=&\frac{k}{2}\displaystyle{\int_{\mathbb{R}^N}}\left(-\frac{\partial}{\partial y_{1}}U^{2}_{\mu_{2}}\left(y-\frac{P_{1}}{\epsilon}\right)\right)\phi_\epsilon(\epsilon y)Y(\epsilon y)dy\\
    =&\frac{k\epsilon}{2}\displaystyle{\int_{\mathbb{R}^N}}\left[\frac{\partial}{\partial y_{1}}\phi_\epsilon(\epsilon y)-\frac{\partial}{\partial y_{1}}\phi_\epsilon(0)\right]Y(\epsilon y)\,U^{2}_{\mu_{2}}\left(y-\frac{P_{1}}{\epsilon}\right)dy\\
    &+\frac{k\epsilon}{2}\displaystyle{\int_{\mathbb{R}^N}}\left[\frac{\partial}{\partial y_{1}}Y(\epsilon y)-\frac{\partial}{\partial y_{1}}Y(0)\right]\phi_\epsilon(\epsilon y)\,U^{2}_{\mu_{2}}\left(y-\frac{P_{1}}{\epsilon}\right)dy\\
    \lesssim&\epsilon\|Y\|_{L^{\infty}(\R^N)}\left\|\phi_\epsilon\right\|_{C^{1,1-\frac{N}{q}}(\R^N)}\displaystyle{\int_{\mathbb{R}^N}}\left|\epsilon y\right|^{1-\frac{N}{q}}\,U^{2}_{\mu_{2}}\left(y-\frac{P_{1}}{\epsilon}\right)dy\\
    &+\epsilon\|\phi_\epsilon\|_{L^{\infty}(\R^N)}\left\|Y\right\|_{C^{2}(\R^N)}\displaystyle{\int_{\mathbb{R}^N}}\left|\epsilon y\right|\,U^{2}_{\mu_{2}}\left(y-\frac{P_{1}}{\epsilon}\right)dy\\\lesssim&\epsilon\rho^{1-\frac{N}{q}}\left\|\phi_\epsilon\right\|_{C^{1,1-\frac{N}{q}}(\R^N)}+\epsilon\rho\left\|\phi_\epsilon\right\|_{V}
    \\
   \lesssim& \begin{cases}
     \epsilon^{1+\frac{2N}{q}}\rho^{1-\frac{N}{q}} ,\quad &q>2N,\nonumber\\
       \epsilon\rho^{1-\frac{N}{q}}\left(e^{-2\sqrt{\omega_{0}}\frac{\rho}{\epsilon}\sin{\frac{\pi}{mk}}}\left(\frac{\rho}{\epsilon}\right)^{-\frac{N-1}{2}}\right) ,\quad &q\in(N,2N).
    \end{cases}\\    =&o\left(\epsilon\rho\right).
\end{align*}
Moreover,
\begin{align*}
   \left|D_5\right|\lesssim \|\phi_\epsilon\|_{L^{\infty}(\R^N)}e^{-2\sqrt{\omega_{0}}\frac{\rho}{\epsilon}\sin{\frac{\pi}{k}}}\left(\frac{\rho}{\epsilon}\right)^{-\frac{N-1}{2}}\lesssim e^{-4\sqrt{\omega_{0}}\frac{\rho}{\epsilon}\sin{\frac{\pi}{mk}}}\left(\frac{\rho}{\epsilon}\right)^{-{(N-1)}}=o\left(\epsilon\rho\right).
\end{align*}
Thus, from the estimates of $D_4$ and $D_5$, we derive that if $\alpha\neq0$,
\begin{align*}
    |P_{10}|=o\left(\epsilon\rho\right).
\end{align*}

In addition, {using similar arguments of $P_{10}$ to calculate $P_{11}$, we obtain that either $\alpha=0$ or $\alpha\neq0$},
\begin{align*}
|P_{11}|=o\left(\epsilon\rho\right).
\end{align*}

{Now, we estimate the last two terms in \eqref{L-2}, which vanish if $\alpha=0$}. As for $P_{12}$, taking into account \eqref{bala2}, we derive that if $\alpha\neq0$
\begin{align*}
    |P_{12}|&\lesssim\|\psi_\epsilon\|_{L^{\infty}(\R^N)}\\
&\cdot\left(\displaystyle\int_{\R^N}U_{\rho,\epsilon}^{2}\,U_{\rho,\epsilon}^{2}\left(\widehat{\mathcal{R}_{2}}    y\right)+U_{\rho,\epsilon}^{2}\,U_{\rho,\epsilon}\left(\widehat{\mathcal{R}_{2}}    y\right)\,\Psi_{\epsilon}\left(\widehat{\mathcal{R}_{2}}    y\right)+U_{\rho,\epsilon}\,\Psi_{\epsilon}\,U^{2}_{\rho,\epsilon}\left(\widehat{\mathcal{R}_{2}}    y\right)\right)^{\frac{1}{2}}\\
&\lesssim e^{-2\sqrt{\omega_{0}}\frac{\rho}{\epsilon}\sin{\frac{\pi}{mk}}}\left(\frac{\rho}{\epsilon}\right)^{-\frac{N-1}{2}}\\&\cdot\left(\displaystyle\int_{\R^N}U_{{\mu_{2}},P_{\epsilon,1}}^{2}\,U_{{\mu_{2}},P_{\epsilon,12}}^{2}+\|\Psi_\epsilon\|_{L^{\infty}(\R^N)}\displaystyle\int_{\R^N}\left(U_{{\mu_{2}},P_{\epsilon,1}}^{2}\,U_{\mu_{2},P_{\epsilon,12}}+U_{\mu_{2},P_{\epsilon,1}}\,U^{2}_{{\mu_{2}},P_{\epsilon,12}}\right)\right)^{\frac{1}{2}}\\
&\lesssim \begin{cases}
   e^{-2\sqrt{\omega_{0}}\frac{\rho}{\epsilon}\sin{\frac{\pi}{mk}}}\left(\frac{\rho}{\epsilon}\right)^{-\frac{N-1}{2}}\left(e^{-2\sqrt{\omega_0}\frac{|P_{1}-P_{12}|}{\epsilon}}\left(\frac{|P_{1}-P_{12}|}{\epsilon}\right)^{-\frac{1}{2}}\right)^{\frac{1}{2}},\quad&N=2,\\
    e^{-2\sqrt{\omega_{0}}\frac{\rho}{\epsilon}\sin{\frac{\pi}{mk}}}\left(\frac{\rho}{\epsilon}\right)^{-\frac{N-1}{2}}\left(e^{-2\sqrt{\omega_0}\frac{|P_{1}-P_{12}|}{\epsilon}}\left(\frac{|P_{1}-P_{12}|}{\epsilon}\right)^{-{2}}\ln{\left(\frac{|P_{1}-P_{12}|}{\epsilon}\right)}\right)^{\frac{1}{2}},\quad&N=3,
\end{cases}\\
&=o\left(\epsilon\rho\right).
\end{align*}

Similarly,
\begin{align*}
     |P_{13}|\lesssim e^{-4\sqrt{\omega_{0}}\frac{\rho}{\epsilon}\sin{\frac{\pi}{mk}}}\left(\frac{\rho}{\epsilon}\right)^{-{(N-1)}}=o\left(\epsilon\rho\right).
\end{align*}

 {Thus, from all the above estimates, we conclude that \eqref{hope} holds true in both cases $\alpha=0$ and $\alpha\neq0$. Then combining with \eqref{N2}, we get \eqref{c-alpha} and \eqref{conclusion}, respectively}.

\end{proof}

\subsection{Proof of Theorems \ref{mt1} and \ref{mt}: completed}  {
The last step in the reduction procedure consists in finding $d=d_\epsilon$ in \eqref{pj} such that
the R.H.S. of \eqref{c-alpha} and \eqref{conclusion} vanish, respectively.

If $\alpha=0$ $(m=1)$, we need to find $d=d_\epsilon$ in \eqref{pj} such that
\begin{align*}
    \left(\Delta\omega(0)A\epsilon\rho_\epsilon+B_{1}e^{-2\sqrt{\omega_{0}}\frac{\rho_{\epsilon}}{\epsilon}\sin{\frac{\pi}{k}}}\left(\frac{\rho_\epsilon}{\epsilon}\right)^{-\frac{N-1}{2}}\right)\left(1+o(1)\right)=0,
\end{align*}
for some   positive constants $A$ and $B_1$. Clearly this is possible if    $\Delta \omega(0)<0$ and if we choose
$d_\epsilon=\frac{1}{\sqrt{\omega_{0}}\sin{\frac{\pi}{k}}}+o(1)$ as $\epsilon\rightarrow0$.

If $\alpha\neq0$ $(m\geq2)$}, we need to find $d=d_\epsilon$ in \eqref{pj} such that
$$    \left(  \Delta\omega(0)A\epsilon\rho_\epsilon+\alpha B_2e^{-4\sqrt{\omega_{0}}\sin{\frac{\pi}{mk}}\frac{\rho_\epsilon}{\epsilon}}\left(\frac{\rho_\epsilon}{\epsilon}\right)^{-\frac{1}{2}}\right)\left(1+o(1)\right)=0,\quad\ \hbox{if} \,\, N=2,$$
and
 $$\left( \Delta\omega(0)A\epsilon\rho_\epsilon+\alpha B_2e^{-4\sqrt{\omega_{0}}\sin{\frac{\pi}{mk}}\frac{\rho_\epsilon}{\epsilon}}\left(\frac{\rho_\epsilon}{\epsilon}\right)^{-{2}}\ln{\left(\frac{\rho_\epsilon}{\epsilon}\right)}\right)\left(1+o(1)\right)=0,\quad\ \hbox{if}\,\, N=3.$$
for some   positive constants $A$ and $B_2$.
It is immediate to check that it is possible if  $\alpha$ and $\Delta \omega(0)$ have different sign and if we  choose
$d_\epsilon=\frac{1}{2\sqrt{\omega_{0}}\sin{\frac{\pi}{mk}}}+o(1)$ as $\epsilon\rightarrow0$.
This completes the proof.

\appendix

\section{Technical Lemma}
In the appendix, we present some well-known facts and the useful estimates, which will be utilized in our proof.
\begin{Lem}\label{lma2.2} Let $u, v: \mathbb{R}^N \rightarrow \mathbb{R}$ be two positive continuous radial functions such that
$$
u(x) \sim|x|^a e^{-b|x|}, \quad v(x) \sim|x|^{a^{\prime}} e^{-b^{\prime}|x|},\quad\text{as $|x| \rightarrow +\infty$,}
$$
where $a, a^{\prime} \in \mathbb{R}$, and $b, b^{\prime}>0$. We denote $u_{\xi}(x):=u(x+\xi)$, $\xi \in \mathbb{R}^N$. Then the following asymptotic estimates hold true

$(i)$ If $b<b^{\prime}$,
$$
\int_{\mathbb{R}^N} u_{\xi} v \sim e^{-b|\xi|}|\xi|^a,\quad \text { as }|\xi| \rightarrow+\infty.
$$
A similar expression can be derived  if $b>b^{\prime}$, by replacing $a$ and $b$ with $a^{\prime}$ and $b^{\prime}$.

$(ii)$ If $b=b^{\prime}$, suppose that $a \geq a^{\prime}$. Then

$$
\int_{\mathbb{R}^N} u_{\xi} v \sim \begin{cases}e^{-b|\xi|}|\xi|^{a+a^{\prime}+\frac{N+1}{2}},\quad & \text { if } a^{\prime}>-\frac{N+1}{2}, \\ e^{-b|\xi|}|\xi|^a \log |\xi|,\quad & \text { if } a^{\prime}=-\frac{N+1}{2}, \\ e^{-b|\xi|}|\xi|^a ,\quad& \text { if } a^{\prime}<-\frac{N+1}{2} .\end{cases}
$$

\end{Lem}

\begin{Lem}
    \label{derivative} Let $U$ be the solution of \eqref{U-1} with $\omega_0=\mu_2=1$. Consider the following integral
$$
\Theta_{s, t}(\xi):=\int_{\mathbb{R}^N} U_{\lambda, \mu}^s(x+\xi) \partial_{x_1} U_{\lambda, \mu}^t(x) d x, \quad \xi \in \mathbb{R}^N,
$$
where $U_{\lambda, \mu}(x)=\sqrt{\frac{\lambda}{\mu}} U(\sqrt{\lambda} x)$ and $s, t \geq 1$. Then the following asymptotic estimates hold true

$(i)$ If $s<t$, then
$$
\Theta_{s, t}(\xi) \sim {cs} \frac{\xi_1}{|\xi|} e^{-s \sqrt{\lambda}|\xi|}|\xi|^{-s \frac{N-1}{2}},\quad \text { as }|\xi| \rightarrow+\infty.
$$

$(ii)$ If $s=t$, then

$$
\Theta_{s, t}(\xi) \sim \begin{cases}{cs} \frac{\xi_1}{|\xi|} e^{-s \sqrt{\lambda}|\xi|}|\xi|^{-s(N-1)+\frac{N+1}{2}},\quad & \text { if } s<\frac{N+1}{N-1}, \\ {cs} \frac{\xi_1}{|\xi|} e^{-s \sqrt{\lambda}|\xi|}|\xi|^{-s \frac{(N-1)}{2}} \ln |\xi|,\quad & \text { if } s=\frac{N+1}{N-1} ,\\ {cs} \frac{\xi_1}{|\xi|} e^{-s \sqrt{\lambda}|\xi|}|\xi|^{-s \frac{(N-1)}{2}} ,\quad& \text { if } s>\frac{N+1}{N-1},\end{cases}
$$
for some constants $c>0$. In particular, when $N=1$, the first option in $(ii)$ holds, that is,

$$
\Theta_{s, s}(\xi) \sim cs \frac{\xi_1}{|\xi|} e^{-s\sqrt{\lambda}|\xi|}|\xi|,\quad \text { as }|\xi| \rightarrow+\infty.
$$

\end{Lem}

\section{A regularity result}
\renewcommand{\theequation}{B.\arabic{equation}}
\setcounter{equation}{0}
 In the following, we aim to improve the regularity of the remainder term $\phi_\epsilon$ of the first solution component obtained in \eqref{alnot-zero} in Proposition \ref{contraction}, which will be  crucial in estimating the terms $P_{10}$ and $P_{11}$ given in \eqref{L-2}.
\begin{Lem} \label{imp} Let $N=2,3$. Then for any $q>N$,
    \begin{align}
 \left\|\phi_\epsilon\right\|_{C^{1,1-\frac{N}{q}}(\R^N)}\lesssim\left\|\phi_\epsilon\right\|_{W^{2,q}(\R^N)} \lesssim\begin{cases}
       \epsilon^{\frac{2N}{q}} ,\quad &q>2N,\nonumber\\
       e^{-2\sqrt{\omega_{0}}\frac{\rho}{\epsilon}\sin{\frac{\pi}{mk}}}\left(\frac{\rho}{\epsilon}\right)^{-\frac{N-1}{2}},\quad &q\in(N,2N).
    \end{cases}
\end{align}
\end{Lem}
\begin{proof}
  From Lemma \ref{contraction} and \eqref{first}, we know that there exists $\left(\phi_\epsilon,\psi_{\epsilon}\right)\in K^\bot$ solving
    \begin{align}\label{NEW-EQU}
        \mathcal{L}_{1}\left(\phi_\epsilon,\psi_{\epsilon}\right)=\mathcal{E}_{1}+\mathcal{N}_{1}\left(\phi_\epsilon,\psi_{\epsilon}\right),
    \end{align}
and $$\|(\phi_\epsilon, \psi_\epsilon)\|=O\left(e^{-2\sqrt{\omega_{0}}\frac{\rho}{\epsilon}\sin{\frac{\pi}{mk}}}\left(\frac{\rho}{\epsilon}\right)^{-\frac{N-1}{2}}\right).$$
Moreover, using Sobolev embedding, 
we derive that 
\begin{align*}
   \left\|\phi_\epsilon\right\|_{L^{\infty}(\R^N) } \lesssim O\left(e^{-2\sqrt{\omega_{0}}\frac{\rho}{\epsilon}\sin{\frac{\pi}{mk}}}\left(\frac{\rho}{\epsilon}\right)^{-\frac{N-1}{2}}\right),\quad\left\|\psi_\epsilon\right\|_{L^{\infty}(\R^N) } \lesssim O\left(e^{-2\sqrt{\omega_{0}}\frac{\rho}{\epsilon}\sin{\frac{\pi}{mk}}}\left(\frac{\rho}{\epsilon}\right)^{-\frac{N-1}{2}}\right).
\end{align*}

Now, we improve the regularity of $\phi_\epsilon$ by utilizing the potential theory. Indeed, from \eqref{NEW-EQU}, we have
\begin{align}\label{new-phi}
   -\Delta \phi_\epsilon+V(y) \phi_\epsilon=&\left(3 \mu_1 \Upsilon_{\epsilon}^2(y)+\beta \sum\limits_{i=1}^{m}\Theta_{\epsilon}^2\left(\widehat{\mathcal{R}}_{i}   \frac{y}{\epsilon}\right)\right) \phi_\epsilon(y)\nonumber\\
&+2 \beta \Upsilon_{\epsilon}(y)\left(\sum_{i=1}^{m}\Theta_{\epsilon}\left(\widehat{\mathcal{R}}_{i}   \frac{y}{\epsilon}\right)  \psi_\epsilon\left(\widehat{\mathcal{R}}_{i}   \frac{y}{\epsilon}\right)\right)+\mathcal{E}_{1}+\mathcal{N}_{1}\left(\phi_\epsilon,\psi_{\epsilon}\right)\nonumber\\
:=&F(y).
\end{align}
Then {for any $q>N$},
\begin{align}\label{new-estimate}
    \left\|F(y)\right\|_{L^{q}(\R^N)}\lesssim&\left\|\Upsilon_{\epsilon}^2\,\phi_\epsilon\right\|_{L^{q}(\R^N)}+\left\|\sum\limits_{i=1}^{m}\Theta_{\epsilon}^2\left(\widehat{\mathcal{R}}_{i}   \frac{y}{\epsilon}\right)\,\phi_\epsilon\right\|_{L^{q}(\R^N)}\nonumber\\
    &+\left\|\Upsilon_{\epsilon}(y)\left(\sum_{i=1}^{m}\Theta_{\epsilon}\left(\widehat{\mathcal{R}}_{i}   \frac{y}{\epsilon}\right)  \psi_\epsilon\left(\widehat{\mathcal{R}}_{i}   \frac{y}{\epsilon}\right)\right)\right\|_{L^{q}(\R^N)}\nonumber\\
    &+\left\|\mathcal{E}_{1}\right\|_{L^{q}(\R^N)}+\left\|\mathcal{N}_{1}\left(\phi_\epsilon,\psi_{\epsilon}\right)\right\|_{L^{q}(\R^N)}.
\end{align}
Recall the definition of $\mathcal{E}_{1}$ given in \eqref{error}. By the estimates of correction terms $\Phi_\epsilon$ and $\Psi_\epsilon
    $ obtained in Lemma \ref{Phi}, a simple computation yields that {for any $q>N$,}
    \begin{align*}
        \left\|\mathcal{E}_{1}\right\|_{L^{q}(\R^N)}\lesssim&\left\|Y\,\Phi_\epsilon^{2}\right\|_{L^{q}(\R^N)}+\left\|\Phi_\epsilon^{3}\right\|_{L^{q}(\R^N)}+\left\|Y(y) \left(\sum\limits_{i=1}^{m}\Psi^{2}_{\epsilon}\left(\widehat{\mathcal{R}_{i}}   \frac{y}{\epsilon}\right)\right)\right\|_{L^{q}(\R^N)}\\
        &+\left\|Y(y)\left(\sum\limits_{i=1}^{m} U_{\rho,{\epsilon}}\left(\widehat{\mathcal{R}_{i}}   \frac{y}{\epsilon}\right) \Psi_{\epsilon}\left(\widehat{\mathcal{R}_{i}}   \frac{y}{\epsilon}\right)\right)\right\|_{L^{q}(\R^N)}\\
        &+\left\|\sum\limits_{i=1}^{m}\Phi_\epsilon(y)U_{\rho,{\epsilon}}^2\left(\widehat{\mathcal{R}_{i}}   \frac{y}{\epsilon}\right)\right\|_{L^{q}(\R^N)}+\left\|\sum\limits_{i=1}^{m}\Phi_\epsilon(y)\Psi_{{\epsilon}}^2\left(\widehat{\mathcal{R}_{i}}   \frac{y}{\epsilon}\right)\right\|_{L^{q}(\R^N)}\\&+\left\|\sum\limits_{i=1}^{m}\Phi_\epsilon(y)U_{\rho,{\epsilon}}\left(\widehat{\mathcal{R}_{i}}   \frac{y}{\epsilon}\right) \Psi_{\epsilon}\left(\widehat{\mathcal{R}_{i}}   \frac{y}{\epsilon}\right)\right\|_{L^{q}(\R^N)}\\\lesssim&\left\|\Phi_\epsilon\right\|_{L^{\infty}(\R^N)}^{2}+\left\|\Psi_\epsilon\right\|_{L^{\infty}(\R^N)}^{2}+\left\|Y\right\|_{L^{\infty}(\R^N)}\left\|\Psi_\epsilon\right\|_{L^{\infty}(\R^N)}\left(\displaystyle{\int_{\R^N}}U_{\rho,{\epsilon}}^{q}(x)\,\epsilon^{N}dx\right)^{\frac{1}{q}}\\
        &+\left\|\Phi_\epsilon\right\|_{L^{\infty}(\R^N)}\left(\displaystyle{\int_{\R^N}}U_{\rho,{\epsilon}}^{q}(x)\,\epsilon^{N}dx\right)^{\frac{1}{q}}\\
        &+\left\|\Phi_\epsilon\right\|_{L^{\infty}(\R^N)}\left\|\Psi_\epsilon\right\|_{L^{\infty}(\R^N)}\left(\displaystyle{\int_{\R^N}}\Psi_{{\epsilon}}^{q}(x)\,\epsilon^{N}dx\right)^{\frac{1}{q}}\\
        &+\left\|\Phi_\epsilon\right\|_{L^{\infty}(\R^N)}\left\|\Psi_\epsilon\right\|_{L^{\infty}(\R^N)}\epsilon^{\frac{N}{q}}\left(\displaystyle{\int_{\R^N}}U_{\rho,{\epsilon}}^{q}(x)\,dx\right)^{\frac{1}{q}}\\\lesssim& \epsilon^{\frac{2N}{q}}+\epsilon^{N}+{\epsilon^{\frac{N}{q}+\frac{N}{2}}}+\epsilon^{\frac{2N}{q}+\frac{N}{2}}\lesssim\epsilon^{\frac{2N}{q}}.
    \end{align*}

    Furthermore, combining the definition of $\mathcal{N}_{1}$ given in \eqref{perturbed} with \eqref{alnot-zero}, we derive that for any $q>N$,
    \begin{align*}
       &\left\|\mathcal{N}_{1}\left(\phi_\epsilon,\psi_\epsilon\right)\right\|_{L^{q}(\R^N)}\\\lesssim& \left\|\phi_\epsilon^{3}\right\|_{L^{q}(\R^N)}+\left\|\phi_\epsilon^{2}\,\Upsilon_\epsilon\right\|_{L^{q}(\R^N)}+\left\|\Upsilon_\epsilon\,\psi_{\epsilon}^{2}\left(\frac{y}{\epsilon}\right)\right\|_{L^{q}(\R^N)}\\
       &+\left\|\phi_\epsilon\,\Theta_{\epsilon}\left(\frac{y}{\epsilon}\right)\psi_{\epsilon}\left(\frac{y}{\epsilon}\right)\right\|_{L^{q}(\R^N)}+\left\|\phi_\epsilon\,\psi_{\epsilon}^{2}\left(\frac{y}{\epsilon}\right)\right\|_{L^{q}(\R^N)}\\
       \lesssim&\left\|\phi_\epsilon\right\|^{2}_{L^{\infty}(\R^N)}\left(\displaystyle{\int_{\R^N}}\phi_\epsilon^{q}\right)^{\frac{1}{q}}+\left\|\phi_\epsilon\right\|^{2}_{L^{\infty}(\R^N)}\left(\displaystyle{\int_{\R^N}}\Upsilon_\epsilon^{q}\right)^{\frac{1}{q}}\\
       &+\left\|\psi_\epsilon\right\|^{2}_{L^{\infty}(\R^N)}\left(\displaystyle{\int_{\R^N}}\Upsilon_\epsilon^{q}\right)^{\frac{1}{q}}+\left\|\psi_\epsilon\right\|_{L^{\infty}(\R^N)}\left\|\phi_\epsilon\right\|_{L^{\infty}(\R^N)}\left(\displaystyle{\int_{\R^N}}\Theta_\epsilon^{q}(x)\,\epsilon^N\right)^{\frac{1}{q}}\\
       &+\left\|\psi_\epsilon\right\|_{L^{\infty}(\R^N)}\left\|\phi_\epsilon\right\|_{L^{\infty}(\R^N)}\left(\displaystyle{\int_{\R^N}}\psi_\epsilon^{q}(x)\,\epsilon^N\right)^{\frac{1}{q}}\\
       \lesssim&\left\|(\phi_\epsilon, \psi_\epsilon)\right\|^{2}=O\left(e^{-4\sqrt{\omega_{0}}\frac{\rho}{\epsilon}\sin{\frac{\pi}{mk}}}\left(\frac{\rho}{\epsilon}\right)^{-{(N-1)}}\right).
    \end{align*}

For the remaining terms in \eqref{new-estimate}, we have
\begin{align*}
    &\left\|\Upsilon_{\epsilon}^2\,\phi_\epsilon\right\|_{L^{q}(\R^N)}+\left\|\sum\limits_{i=1}^{m}\Theta_{\epsilon}^2\left(\widehat{\mathcal{R}}_{i}   \frac{y}{\epsilon}\right)\,\phi_\epsilon\right\|_{L^{q}(\R^N)}+\left\|\Upsilon_{\epsilon}\left(\sum_{i=1}^{m}\Theta_{\epsilon}\left(\widehat{\mathcal{R}}_{i}   \frac{y}{\epsilon}\right)  \psi_\epsilon\left(\widehat{\mathcal{R}}_{i}   \frac{y}{\epsilon}\right)\right)\right\|_{L^{q}(\R^N)}\\\lesssim&\left\|\phi_\epsilon\right\|_{L^{\infty}(\R^N)}+\left\|\phi_\epsilon\right\|_{L^{\infty}(\R^N)}\left(\displaystyle{\int_{\R^N}}\Theta_\epsilon^{2q}(x)\,\epsilon^N\right)^{\frac{1}{q}}+\left\|\psi_\epsilon\right\|_{L^{\infty}(\R^N)}\epsilon^{\frac{N}{q}}\left\|\Upsilon_\epsilon\right\|_{L^{\infty}(\R^N)}\left\|\Theta_\epsilon\right\|_{L^{q}(\R^N)}\\
    \lesssim&\left\|(\phi_\epsilon, \psi_\epsilon)\right\|=O\left(e^{-2\sqrt{\omega_{0}}\frac{\rho}{\epsilon}\sin{\frac{\pi}{mk}}}\left(\frac{\rho}{\epsilon}\right)^{-\frac{N-1}{2}}\right)
\end{align*}

Therefore, we conclude that for any $q>N$,
\begin{align}
    \left\|F(y)\right\|_{L^{q}(\R^N)}\lesssim&\epsilon^{\frac{2N}{q}}+e^{-2\sqrt{\omega_{0}}\frac{\rho}{\epsilon}\sin{\frac{\pi}{mk}}}\left(\frac{\rho}{\epsilon}\right)^{-\frac{N-1}{2}}\nonumber\\
    =&\begin{cases}
       \epsilon^{\frac{2N}{q}} ,\quad &q>2N,\nonumber\\
       e^{-2\sqrt{\omega_{0}}\frac{\rho}{\epsilon}\sin{\frac{\pi}{mk}}}\left(\frac{\rho}{\epsilon}\right)^{-\frac{N-1}{2}},\quad &q\in(N,2N),
    \end{cases}
\end{align}
then from \eqref{new-phi} and the assumption $\left(\mathbf{V}_2\right)$, we derive that
\begin{align}
    \left\|\phi_\epsilon\right\|_{W^{2,q}(\R^N)}   \lesssim\left\|F(y)\right\|_{L^{q}(\R^N)}\lesssim\begin{cases}
       \epsilon^{\frac{2N}{q}} ,\quad &q>2N,\nonumber\\
       e^{-2\sqrt{\omega_{0}}\frac{\rho}{\epsilon}\sin{\frac{\pi}{mk}}}\left(\frac{\rho}{\epsilon}\right)^{-\frac{N-1}{2}},\quad &q\in(N,2N).
    \end{cases}
\end{align}
In addition, for any $q>N$, we know that $W^{2,q}(\R^N)\hookrightarrow C^{1,1-\frac{N}{q}}(\R^N)$ as Sobolev embedding Theorem, then
\begin{align*}
 \left\|\phi_\epsilon\right\|_{C^{1,1-\frac{N}{q}}(\R^N)} \lesssim\begin{cases}
       \epsilon^{\frac{2N}{q}} ,\quad &q>2N,\nonumber\\
     e^{-2\sqrt{\omega_{0}}\frac{\rho}{\epsilon}\sin{\frac{\pi}{mk}}}\left(\frac{\rho}{\epsilon}\right)^{-\frac{N-1}{2}},\quad &q\in(N,2N).
    \end{cases}
\end{align*}

\end{proof}

\bibliography{system}
\bibliographystyle{abbrv}

\end{document}